\documentclass[12pt]{amsart}

\usepackage{epsfig}
\usepackage{amsmath}
\usepackage{amssymb}
\usepackage{amscd}
\usepackage{graphicx}
\usepackage{xcolor}
\usepackage{float}
\usepackage{verbatim}
\usepackage{bbm}
\usepackage{cancel}
\usepackage{enumerate}
\usepackage[bookmarks=false,hidelinks, colorlinks, citecolor=black,linkcolor=blue]{hyperref}
\usepackage{tikz,pgfplots}
\usetikzlibrary{patterns.meta}
\usetikzlibrary{patterns}
\pgfplotsset{compat=1.17}

\newtheorem{thm}{Theorem}[section]
\newtheorem{lem}[thm]{Lemma}
\newtheorem{prop}[thm]{Proposition}

\newtheorem{ques}[thm]{Question}
\newtheorem{cor}[thm]{Corollary}
\newtheorem{maintheorem}{Theorem}

\newtheoremstyle{remarkbold}
  {6pt plus 2pt minus 2pt} 
  {6pt plus 2pt minus 2pt} 
  {\normalfont}
  {}
  {\bfseries}
  {.}
  {.5em} 
  {}

\theoremstyle{remarkbold}
\newtheorem{defn}[thm]{Definition}
\newtheorem{rem}[thm]{Remark}
\newtheorem{claim}[thm]{Claim}
\newtheorem{exam}[thm]{Example}

\def \N {\mathbb N}
\def \T {\mathbb T}

\def \Z {\mathbb Z}\def\Sep{\mathrm{Sep}}
\def \R {\mathbb R}

\def \Q {\mathbb Q}
\def \U {\mathcal U}
\def \P {\mathcal P}

\def \M {\mathcal M}

\def \xp {(X,\R)}

\DeclareMathOperator{\mdim}{mdim}
\def \sq {sequence}

\def\ocap {\mathrm{ocap}}

\DeclareMathOperator{\id}{Id}
\DeclareMathOperator{\PP}{P}

\def \d {\delta}

\def \pp {\partial^\R }
\DeclareMathOperator{\ip}{\mathrm{Int}^\R}

\def\mc{\mathcal}

\DeclareMathOperator{\per}{P}
\DeclareMathOperator{\stab}{stab}
\DeclareMathOperator{\Int}{Int}

\numberwithin{equation}{section}

\subjclass[2020]{Primary: 37B40.  Secondary: 54F45.}
\keywords{Topological flows, mean dimension,  small boundary property, small flow boundary property, flow-generated entropy, dynamical Cantor staircase function, marker property.}

\thanks{T.D. and Y.G. were partially supported by the National Science Centre (Poland) grant 2020/39/B/ST1/02329.
C.L. was supported by the
Postdoctoral Fellowship Program and China Postdoctoral Science Foundation under Grant Number BX20250067, and the China Postdoctoral Science Foundation under Grant Number 2025M773074.}

\thanks{\textsuperscript{*}Corresponding author.}
\begin{document}

\title[Dynamical Cantor Staircase Functions \& SFBP]{Dynamical Cantor Staircase Functions and The Small Flow Boundary Property}

\author[Tomasz Downarowicz, Yonatan Gutman and Chunlin Liu]{Tomasz Downarowicz, Yonatan Gutman and Chunlin Liu\textsuperscript{*}}

\address{\vskip 2pt \hskip -12pt Tomasz Downarowicz}

\address{\hskip -12pt Faculty of Pure and Applied Mathematics, Wroc\l aw University of Technology, Wroc\l aw, Poland}

\email{downar@pwr.edu.pl}

\medskip
\address{\vskip 2pt \hskip -12pt Yonatan Gutman}

\address{\hskip -12pt Institute of Mathematics, Polish Academy of Sciences, ul. Śniadeckich 8, 00-656 Warszawa, Poland}

\email{gutman@impan.pl}

\medskip
\address{\vskip 2pt \hskip -12pt Chunlin Liu}
\address{\hskip -12pt School of Mathematical Sciences, Dalian University of Technology, Dalian, 116024, P.R. China}
\address{\hskip -12pt Institute of Mathematics, Polish Academy of Sciences, ul. Śniadeckich 8, 00-656 Warszawa, Poland}

\email{chunlinliu@mail.ustc.edu.cn}

\begin{abstract}
The small flow boundary property (SFBP), introduced by Burguet for fixed-point free topological flows, is a non-trivial generalization of the small boundary property (SBP).
We  characterize  when a time-discretization of such a flow satisfies the SBP and deduce that an SFBP flow admitting an aperiodic time-discretization has vanishing mean dimension.

Furthermore, we introduce a new quantity, \textit{flow-generated entropy}, for a factor between a flow and  a time-discretization, quantifying the dynamical complexity inherited from the flow itself. This is used in order to establish that any time-discretization of a  flow with SFBP admits factors of arbitrarily small flow-generated entropy separating any fixed pair of distinct points. The argument relies on a construction of a dynamical version of the Cantor staircase function. 

Finally, the appendix includes proofs of fundamental properties of the marker property which have not yet appeared in the literature. 

\end{abstract}

\maketitle
\tableofcontents

\section{Introduction}
In 1991 Shub and Weiss (\cite{SW91}) introduced the notion of \textit{small} sets for topological dynamical systems. For closed sets these are exactly the sets whose measure with respect to any invariant measure is zero. Subsequently Lindenstrauss and Weiss (\cite{LW}) investigated topological dynamical systems with the \textit{small boundary property} (SBP). These are systems admitting a basis of open sets whose boundaries are small. Arguably this is a natural dynamical analog of the zero-dimensional spaces, spaces which play an important role in many branches of mathematics such as topology, topological dynamics, descriptive set theory, and ergodic theory. It is therefore perhaps not surprising that SBP has attracted increasing attention and has been found to have profound applications (see, e.g., \cite{L99, LW, BDz2004, Dz2005, GutmanLindenstraussTsukamoto2016,gutman2017embedding,kerr2020almost,KerrKopsacheilisPetrakos2024,KopsacheilisLiaoTikuisisVaccaro2025}). In particular, SBP implies vanishing mean dimension (\cite{LW}) and representability as an inverse limit of finite entropy systems (\cite{L99}). Whether the two latter properties are equivalent remains one of the outstanding open problems in the theory of mean dimension.

Alongside discrete-time $\mathbb{Z}$-topological dynamical systems, continuous-time $\mathbb{R}$-topological dynamical systems, known as \textit{topological flows}, occupy a distinguished position. Indeed, many dynamical systems encountered in nature evolve continuously and are therefore naturally modeled by an $\mathbb{R}$-action. Given the growing importance of the small boundary property, it is natural to investigate SBP in the context of topological flows (and perhaps more generally, continuous group actions). The first obstacle one encounters, however, is how to formulate an appropriate definition. 

It is not difficult to see that a direct generalization of SBP to the $\R$-setting does not yield an appropriate definition. Indeed, any set that is ``transversal" to the flow direction has measure zero with respect to every $\R$-invariant measure. Thus the class of flows admitting bases of open sets with boundaries having this property seems overly broad. In \cite{burguet2019symbolic}, Burguet addressed this difficulty by introducing the \textit{small flow boundary property} (SFBP) for fixed-point free topological flows. The definition requires any closed cross-section of the flow  to admit a basis of (relatively) open sets whose generated   \textit{flow-boxes} have small boundaries. This crucially includes the part of the boundary which is aligned with the flow. Using the fact that this requirement must hold for every closed cross-section one can show that a topological dynamical system has SBP if and only if any suspension
flow over it (under an arbitrary roof function) has SFBP (\cite[Lemma~2.15]{burguet2019symbolic}). 

Beyond suspension flows, the small flow boundary property poses formidable challenges and no general characterization is known. A natural approach is to study the relation between SBP and SFBP in terms of the  $\tau$-discretizations  of a flow $(X,\R)$, that is, the induced (discrete) topological dynamical system $(X,\tau)$ corresponding to translations by $\tau\in \R$. 
In this vein,
for flows with SFBP, we characterize when $(X,\tau)$ has SBP:
 \begin{maintheorem}\label{thm:SFBPvsSBP_tau}
    Let $(X,\R)$ be a topological flow with SFBP. Then for any $\tau \in \mathbb{R}$, the following are equivalent:
    \begin{enumerate}
        \item $(X,\tau)$ is aperiodic.
        \item $(X,\tau)$ has SBP.
    \end{enumerate}
\end{maintheorem}

As SBP implies vanishing mean dimension it follows that suspension flows with SFBP have vanishing mean dimension (see Proposition~\ref{prop:sfbp_suspension_mdim_zero}). The validity or
falsity of this statement for general topological flows  with SFBP remains an open problem. However as vanishing mean dimension for some time-discretization of a flow implies the flow itself has vanishing mean dimension we may deduce from Theorem \ref{thm:SFBPvsSBP_tau} the following result:

\begin{maintheorem}\label{thm:SFBP_implies_mdim0}
A topological flow $(X,\R)$ with SFBP, admitting some aperiodic time-discretization, satisfies  $\mdim(X,\R)~=~0.$
\end{maintheorem}

From Theorem \ref{thm:SFBPvsSBP_tau} it follows that an aperiodic $\tau$-discretization $(X,\tau)$ of a  topological flow $(X,\R)$ with SFBP, may be represented as an inverse limit
of finite entropy systems. Suppose now $\tau \in \R\setminus \{0\}$ such that $(X,\tau)$ is not aperiodic. Does $(X,\tau)$ still have a meaningful representation as an inverse limit? We give a positive answer involving a new continuous-discrete invariant of extensions.  Let $\pi:(X,\tau)\to (Y,S)$  be a factor map.
The \emph{flow-generated entropy} of $\pi$ is defined by
\[
h_{\R}(\pi):=\sup_{\mu\in \PP_\R(X)} h_{\pi_*\mu}(S),
\]
where $\PP_\R(X)$ denotes the set of all $\R$-invariant Borel probability measures on $X$. Thus, this quantity quantifies how much of the dynamical complexity of the factor is inherited from the ambient flow.
We prove a general separation theorem and deduce a inverse limit representation from it.

\begin{maintheorem}\label{thm:SFBP_small_entropy_factors}
Let $(X,\R)$ be a topological flow with SFBP, and fix $\tau\in\R$.
Then 
\begin{enumerate}
    \item  $(X,\tau)$ admits arbitrarily small flow-generated entropy factors separating points:
for every $\delta>0$ and any two distinct points $w,w'\in X$, there exists a factor map $\pi:(X,\tau)\to (Y,T)$
such that
$
\pi(w)\neq \pi(w')
$
and $h_{\R}(\pi)<\delta$.
\item 
There exists a compatible system of factor maps $\pi_n: (X,\tau)\rightarrow (W_n,T_n)$, so that  $h_{\R}(\pi_n)<\infty$ for all $n$ and $(X,\tau)=\varprojlim (W_n,T_n)$.
\end{enumerate}

\end{maintheorem}

The proof of Theorem \ref{thm:SFBP_small_entropy_factors}
is based on the construction of dynamical version of the Cantor staircase function (see Definition \ref{defn:dyn_Cantor}). This is a continuous function from the space onto the interval, separating two pre-chosen, distinct points, such that up to measure zero with respect to any flow-invariant measure, takes only dyadic values. Thus, similarly to the classical Cantor staircase function, although it assumes only countably many values from a measure-theoretic perspective, it still forms a continuous path from $0$ to $1$ in its image. Accordingly, one may anticipate that this construction will be useful in future applications.

\medskip

\noindent\textbf{Acknowledgement.}
We thank Ruxi Shi for many insightful discussions. In particular, precursors to Theorems \ref{thm:SFBPvsSBP_tau} and \ref{thm:SFBP_implies_mdim0} for aperiodic flows were first obtained  by Ruxi Shi and the second author. We are  grateful to Michael Levin for suggesting Example~\ref{ex1}. We also thank Wen Huang, Leiye Xu, and Guohua Zhang for helpful discussions related to this work.

\section{Preliminaries}\label{sec:pre}
\subsection{Dynamical systems}
A \textbf{$G$-topological dynamical system} ($G$-t.d.s.\ for short) is a pair $(X,G)$, where $X$ is a compact metrizable space and the group $G$ acts on $X$ by homeomorphisms. The action of an element $g\in G$ on an element $x\in X$ is denoted by $gx$.
In this paper we mainly consider the cases $G=\Z$ and $G=\R$.
We call a $\Z$-t.d.s. simply a \textbf{t.d.s.} and write it as $(X,T)$, where $T$ denotes the generator of the $\Z$--action.
We call an $\R$-t.d.s. a \textbf{topological flow} and write it as $(X,\R)$.

For a flow $(X,\R)$ and $\tau\in\R$, we denote by $(X,\tau)$ the $\Z$-system given by the \textbf{$\tau$-discretization} of the flow, i.e.\ the t.d.s.\ generated by the time-$\tau$ map.  Moreover,  for real numbers $s'>s$ and subset $A$ of $X$, we denote
 \[sA:=\{sa:a\in A\},\]
 and 
 \[[s,s']A:=\bigcup_{t\in [s,s']}tA.\]

Let $(X,G)$ be a $G$-t.d.s. A point $x \in X$ is called a \textbf{fixed-point} if $g(x)=x$ for all $g\in G$, and is a \textbf{periodic point} if there exists $g\in G\setminus\{e\}$ such that $gx=x$. 
When $G=\Z$, for all $n\in\N$,  denote 
\[\per_n(X,T):=\{x\in X: T^n(x)=x\}\text{ and }\per(X,T):=\cup_{n\in\N}\per_n(X,T).\]
The $G$-t.d.s. $(X, G)$ is said to be \textbf{aperiodic} if it has no periodic points.

A continuous surjective map $\psi:(X,G)\to(Y,G)$ between two $G$-t.d.s. is called an \textbf{extension} or \textbf{a factor map} if $\psi(gx)=g\psi(x)$ for  all $g\in G$. We refer to $(X,G)$ as the extension of $(Y,G)$, and $(Y,G)$ as the factor of $(X,G)$.

Given a sequence of $G$-t.d.s.	$\{(X_i,G)\}_{i=1}^\infty$  and factor maps $\pi_{i+1}:X_{i+1}\to X_i$ for $i=1,2, \ldots$,  we define the \textbf{inverse limit}:
		\[\lim
		\limits_{\longleftarrow}(X_i,G)=\{(x_1,x_2,\ldots)\in \prod_{i=1}^\infty X_i:x_i\in X_i\text{ and }\pi_{i+1}(x_i)=x_i\},\]
		with the product topology and the diagonal $G$-action.

Let $(X,G)$ be a $G$-t.d.s.
A Borel probability measure $\mu$ on $X$ is called \textbf{$G$-invariant} if $g_*\mu=\mu$ for every $g\in G$, where
\[
g_*\mu(A):=\mu(g^{-1}A)\qquad \text{for all Borel sets }A\subset X.
\]
In the case  where $G=\R$, these  measures are also referred to as \textbf{flow-invariant}.
We write $\PP_G(X)$ for the set of all $G$-invariant Borel probability measures on $X$.
A measure $\mu\in \PP_G(X)$ is \textbf{ergodic} if every $G$-invariant Borel set $A\subset X$ satisfies $\mu(A)\in\{0,1\}$.
We denote by $\PP_G^e(X)$ the set of all ergodic measures in $\PP_G(X)$.
In the case $G=\Z$, we abbreviate $\PP_T(X):=\PP_\Z(X)$ and $\PP_T^e(X):=\PP_\Z^e(X)$. Whenever a measure $\nu$ is fixed, we implicitly work with the $\nu$-completion of the
Borel $\sigma$-algebra of $X$; in particular, we may consider $\nu(A)$ for  $A$ in this completion.

For any bounded function $f:X\to \R$, we denote
\[\|f\|_\infty:=\sup_{x\in X}|f(x)|.\]

\subsection{Suspension flows}
Now we recall a classical method to obtain a topological flow from a t.d.s. Let $(X, T)$ be a t.d.s., and $f: X \rightarrow \mathbb{R}_{>0}$ be a positive continuous function. 

Consider the quotient space
$$
S_f X := (X\times \R)/\sim,\qquad (x,s+f(x))\sim (Tx,s).
$$
This quotient space $S_f X$ is compact and metrizable. The \textbf{suspension flow} over $(X, T)$ under the roof function $f$ is the $\R$-t.d.s. on $S_f X$ induced by $t(x, s)=(x,s+t)$.

\begin{defn}
  Two topological flows
$(X,\R)$ and $(Y,\R)$ are said to be \textbf{topologically orbit
equivalent} if there exists a homeomorphism $h:X\to Y$ mapping orbits onto orbits, preserving their orientation. 
\end{defn}
The following fact is well-known.

\begin{lem}
\label{lem:orbit_eq}
Let $(X,T)$ be a t.d.s.\ and
$f:X\to \R_{>0}$ a continuous function. Then $(S_f X,\R)$ is topologically orbit
equivalent to $(S_1 X,\R)$,  the suspension flow over $(X,T)$ under the constant roof function $1$.
\end{lem}
\subsection{Cross-sections and flow-boxes}\label{sec:flow-box}
Let $(X,\R)$ be a topological flow.
 	A \textbf{cross-section with injectivity time $\eta>0$} is a subset $S\subset X$ such that the map \begin{equation}\label{eq:2026311626}
 	   [-\eta,\eta]\times S\to X,\ (t,x)\mapsto tx
 	\end{equation} 
    is injective.
 In this case, we set $L_\eta:=[-\eta,\eta](S)$ and call $L_\eta$ the   \textbf{flow-box associated to $S$}. In particular, for any $ z \in L_\eta $, there exists a unique pair $ (t, x) \in [-\eta, \eta] \times S $ such that $ z = tx $. Hence we may identify $L_\eta$ with
$S\times[-\eta,\eta]$ via \eqref{eq:2026311626}.
    
   A cross-section is said to be \textbf{global} if there exists $\xi>0$ such that $$[-\xi,\xi]S=X.$$ 
The existence of global cross-sections for fixed-point free flows is well known (see  \cite[Lemma~7]{bowen1972expansive} and  \cite[Lemma~5.1]{gutman2024strongly}).
 \begin{prop}\label{prop:existence of complete cross-section}
     Let $\xp$ be a fixed-point free topological flow. Then it admits a
 global closed cross-section.
 \end{prop}
 \subsection{Flow boundaries and interiors}
Let $(X,\R)$ be a topological flow.
  \begin{defn}
Let $V\subset S \subset X$. Denote by $\partial_S V$ and $\Int_S V$ respectively the boundary and interior of $V$ with respect to the subspace topology induced by $S$.
 \end{defn}

 \begin{defn}{(\cite[Definition 2.2]{burguet2019symbolic})}
 Let $U$ be a set contained in a closed cross-section $S$ with injectivity time $\eta$. The \textbf{flow interior} $\ip(U)\subset U$ is defined as
 \[\ip U:=\mathrm{Int}([-\eta,\eta]U)\cap U.\]
 The \textbf{flow boundary} $\pp U$ of $U$ is defined as 
  \[\pp U:=\overline{U}\setminus \ip U.\]
 \end{defn}
\begin{rem}\label{rem:equiv def. of flow b}
Let $S$ be a closed cross-section with injectivity time $\eta$, and let
$U\subset S$. By \cite[Definition~2.7]{gutman2024strongly}, the flow interior
$\ip U\subset U$ is the unique set satisfying
\[
(-\eta,\eta)(\ip U)=\Int([-\eta,\eta]U).
\]
Similarly\footnote{This corrects a typo in
\cite[Definition~2.7]{gutman2024strongly}.}, the flow boundary $\pp U$ is the unique set satisfying
\[
\partial([-\eta,\eta]U)
=
(-\eta\,\overline U)\ \sqcup\ (\eta\,\overline U)\ \sqcup\ ((-\eta,\eta)\pp U).
\]
\end{rem}

\begin{lem}{(\cite[Lemma 2.3]{burguet2019symbolic})}
\label{lem:flow boudary}
     Let $U$ and $V$ be two subsets of a closed cross-section. Then 
     \begin{enumerate}
         \item $\pp(U\cup V)\subset \pp U\cup\pp V;$
         \item $\ip U\cup\ip V\subset  \ip(U\cup V)$.
     \end{enumerate}
\end{lem}

The following results \cite[Proposition 2.15 \& Remark 2.16]{gutman2024strongly} allow us, in many situations, to replace the flow boundary and flow interior
by the ordinary boundary and interior in the subspace topology of a cross-section, which
significantly simplifies several arguments.
\begin{prop}\label{prop:subspace-vs-flow-boundary}
    Let $(X, \R)$ be a topological flow. Let $U$ be a subset of a closed cross-section $S$ such that $\overline{U} \cap \pp S=\emptyset$. Then
$$
\pp U=\partial_S U \text { and } \ip U=\Int_S U.
$$
\end{prop}

\subsection{Topological and measure-theoretic entropy}
Let $(X,T)$ be a t.d.s.\ and let $\rho$ be a continuous metric on $X$.
For $n\in\N$, set $$\rho_n(x,y):=\max_{0\le k\le n-1}\rho(T^k x,T^k y).$$
Given $n\in\N$ and $\epsilon>0$, let $\Sep_\rho(n,\epsilon)$ be the maximal cardinality of an $(n,\epsilon,\rho)$-separated set, and define
\[
h_\rho(T,\epsilon):=\limsup_{n\to\infty}\frac1n\log \Sep_\rho(n,\epsilon).
\]
Define  \textbf{topological entropy} of the t.d.s. $(X,T)$ by
\[\mathrm{h_{top}}(X,T):=\lim_{\epsilon\to0}\mathrm{h_{\rho}}(T,\epsilon).\]

Let $\mu\in\mc \PP_T(X)$, and $\alpha=\{A_1,\dots,A_k\}$ be a finite measurable partition of $X$.
Define its \textbf{(Shannon) entropy} by
\[
H_\mu(\alpha)
:=
-\sum_{i=1}^k \mu(A_i)\log \mu(A_i).
\]
Define its \textbf{(dynamical) entropy} by
\[
\mathrm{h_{\mu}}(T,\alpha)
:=\lim_{n\to\infty}\frac{1}{n}H_\mu(\alpha_0^{n-1}),
\]
and the  \textbf{(measure-theoretic) entropy}  of
$T$ by
\[
\mathrm{h_{\mu}}(T):=\sup_{\alpha} \mathrm{h_{\mu}}(T,\alpha),
\]
where the supremum is over all finite measurable partitions $\alpha$ of $X$.

By the variational principle \cite[Theorem 8.6]{W82},
\[
\mathrm{h_{top}}(X,T)=\sup_{\mu\in\mathcal \PP_T(X)} \mathrm{h_{\mu}}(T)
= \sup_{\mu\in\mathcal \PP_T^e(X)} \mathrm{h_{\mu}}(T).
\]

Given a topological flow  $(X,\R)$,  it follows from \cite[Proposition 21]{B71}  that for  all $\tau\in\mathbb{R},$
\[\mathrm{h_{top}}(X,\tau):=|\tau|\cdot \mathrm{h_{top}}(X,1). \]
Thus, we define the topological entropy of the topological flow by
\[\mathrm{h_{top}}(X,\R)=\mathrm{h_{top}}(X,1).\]

\subsection{Lebesgue covering dimension and mean dimension}
 Let $(X,d)$ be a compact metric space, $\epsilon>0$, and $Y$ be a topological space. Denoted by $\dim (X)$ the \textbf{Lebesgue covering dimension} of $X$. A continuous map $f:X\to Y$ is called an $\epsilon$-embedding if for  all $x_1,x_2\in X$ with $f(x_1)=f(x_2)$, we have $d(x_1,x_2)<\epsilon$. We define 
\[\mathrm{Widim}_\epsilon(X,d)=\min_{K\in\mathcal{K}}\mathrm{dim}(K).\]
where $\mathcal{K}$ denotes the collection of compact metrizable spaces $K$ satisfying that there exists an $\epsilon$-embedding $f:X\to K$.

Following \cite[Theorem 2.6]{LW},  we define the \textbf{mean dimension } of a t.d.s. $(X,T)$ by 
    \[\mathrm{mdim}(X,T):=\lim_{\epsilon\to0}\lim_{N\to\infty}\frac{\mathrm{Widim}_\epsilon(X,d_N)}{N}.\]

Recently, Gutman and Jin   \cite{GutmanJin2020} introduced the \textbf{mean dimension} of a flow $(X,\R)$ by 
    \[\mathrm{mdim}(X,\R):=\lim_{\epsilon\to0}\lim_{N\to\infty}\frac{\mathrm{Widim}_\epsilon(X,d_N)}{N},\]
where $d_N(x,y):=\max_{t\in[0,N]}d(tx,ty)$.

The following result is obtained by combining Proposition 2.5 and Proposition 2.7 in \cite{GutmanJin2020}.
\begin{prop}\label{prop:mdim time-1}
Let $(X,\R)$ be a topological flow. Then 
\[\mathrm{mdim}(X,\tau)=|\tau|\cdot\mathrm{mdim}(X,\R).\]
\end{prop}

\subsection{The marker property}\label{sec:marker}
\begin{defn}\label{defn:marker}
Let $(X,T)$ be a t.d.s. We say that $(X,T)$ has the \textbf{marker property} if for every finite set $F\subset \Z\setminus\{0\}$ there exists an open set $U\subset X$ such that
\[
X=\bigcup_{n\in\Z}T^nU
\quad\text{and}\quad
U\cap T^nU=\emptyset\ \ \text{for all }n\in F.
\]
\end{defn}

\begin{lem}\label{lem:important observation}
Let $(X,T)$ be a t.d.s., and let $(S_1X,\R)$ be the suspension flow over $(X,T)$ under the constant roof function $1$.
Then for every $\tau\in\R\setminus\Q$, the $\tau$-discretization $(S_1X,\tau)$ has the marker property.
\end{lem}

\begin{proof}
Represent points of $S_1X$ as equivalence classes $(x,t)$ with $t\in[0,1)$.
For $\tau\in\R$ one has
\[
\tau((x,t))
=\bigl(T^{\lfloor t+\tau\rfloor}x,\; t+\tau-\lfloor t+\tau\rfloor\bigr).
\]
Let $R_\tau:\T^1\to\T^1$ be the rotation $s\mapsto s+\tau$ and consider the projection
\[
\pi_1:S_1X\to\T^1,\qquad \pi_1((x,t))=t.
\]
A direct computation shows $\pi_1\circ\tau=R_\tau\circ\pi_1$, so $\pi_1$ is a factor map from $(S_1X,\tau)$ onto $(\T^1,R_\tau)$.

Assume $\tau\notin\Q$. Then $(\T^1,R_\tau)$ is minimal, and 
hence $(\T^1,R_\tau)$ has the marker property.
Finally, since the marker property is stable under extensions, $(S_1X,\tau)$ has the marker property.
\end{proof}

\begin{prop}\label{prop:time-dependence-mp}
There exists an aperiodic topological flow $(X,\R)$ with entropy zero
and two times $t_1\neq t_2$ such that $(X,t_1)$ has the marker property while $(X,t_2)$ does not.
\end{prop}

\begin{proof}
Let $(X,T)$ be an aperiodic t.d.s.\ with entropy zero but without the marker property
(such systems are constructed by Dranishnikov and Levin in \cite{DL25}).
Consider the suspension flow  $(S_1X,\R)$ under the constant roof function $1$.
Note that $\mathrm{h_{top}}(S_1X,\R)=\mathrm{h_{top}}(X,T)=0$\footnote{We may view the t.d.s.\ $(X\times[0,1],\,T\times\id)$, which clearly has entropy zero, as an extension of the
time-one system $(S_1X,1)$.}.
Fix $t_1\in\R\setminus\Q$. By Lemma~\ref{lem:important observation}, $(S_1X,t_1)$ has the marker property.

We claim that $(S_1X,1)$ does not have the marker property.
Using the map $x\mapsto (x,0)$, we  view $(X,T)$ as a subsystem of $(S_1X,1)$.
If $(S_1 X, 1)$ had the marker property, then its subsystem $(X, T)$ would also have the marker property, which is a contradiction.
Set $t_2=1$. Thus $(S_1X, t_2)$ does not have the marker property.
\end{proof}

\subsection{Strongly isomorphic extensions}
Let $(X,G)$ be a $G$-t.d.s.
A Borel set $A\subset X$ is called \textbf{null} if $\mu(A)=0$ for every $G$--invariant Borel probability measure $\mu$ on $X$.
A set is called \textbf{full} if its complement is null.

\begin{defn}\label{df:principal}
Let $\pi:(X,G)\to(Y,G)$ be a factor map between two $G$-t.d.s.
We say that $\pi$ is \textbf{strongly isomorphic} if there exists a $G$-invariant full Borel set $E\subset Y$
such that the restriction $\pi:\pi^{-1}(E)\to E$ is one-to-one.
\end{defn}

\begin{lem}\label{lem:strong_composition}
Let $(X,T_3)$, $(Y,T_2)$ and $(Z,T_1)$ be topological dynamical systems.
Suppose that $p:(X,T_3)\to(Y,T_2)$ and $q:(Y,T_2)\to(Z,T_1)$ are strongly isomorphic.
Then $q\circ p:(X,T_3)\to(Z,T_1)$ is strongly isomorphic.
\end{lem}
\begin{proof}
Let $E$ be a $T_2$-full set such that $p: p^{-1}(E) \to E$ is bijective and let $F$ be a $T_1$-full set such that $q: q^{-1}(F) \to F$ is bijective. Note $E\cap q^{-1}(F)$ is a Borel set and therefore by Suslin's theorem (\cite[Theorem 2.8(2)]{G03}) $R:=q(E\cap q^{-1}(F))$ is Borel. Moreover the complement of $R$ is $T_1$-null. Trivially, $q\circ p:(q\circ p)^{-1}(R)\to R $ is bijective.
\end{proof}

\begin{lem}\label{lem:time 1}
  Let $\pi:(X,\R)\to(Y,\R)$ be a strongly isomorphic factor map between flows.
Fix $\tau>0$. Then $\pi:(X,\tau)\to(Y,\tau)$ is a strongly isomorphic factor map.
\end{lem}
\begin{proof}
   Let $E\subset Y$ be a  $\R$-invariant full set  such that $\pi: \pi^{-1}(E) \to E$ is one-to-one. We will show that $E$ is also a $\tau$-full set. Fix a $\tau$-invariant probability measure $\mu$. It is sufficient to show $\mu(E)=1$. Let 
   $\nu:=\frac{1}{\tau}\int_0^\tau s_* \mu ds.$ Then $\nu$ is $\R$-invariant. It follows that $\nu(E)=1$. Thus for $\mathrm{Leb}$-a.e $s\in[0,\tau]$, $s_*\mu(E)=1.$ Since $E$ is $\R$-invariant, we
conclude that $\mu(E)=1$. 
\end{proof}

\section{SBP and SFBP}\label{sec:SBP vs SFBP}

\subsection{The small boundary property}
Let $(X,T)$ be a t.d.s. and $C\subset X$. Define the \textbf{orbit capacity} of $C$ by
\begin{equation}\label{eq:ocap}
    \mathrm{ocap}(C):=\lim_{n\to\infty}\frac{1}{n}\max_{x\in X}\sum_{i=0}^{n-1}1_C(T^ix)=\inf_{n\in\mathbb{N}}\frac{1}{n}\max_{x\in X}\sum_{i=0}^{n-1}1_C(T^ix).
\end{equation}
 This definition originates in Shub and Weiss \cite{SW91} (see also \cite{L99}). It is straightforward to verify that for  all closed sets $C$, 
\begin{equation}\label{eq123}
    \ocap(C)=\sup_{\mu\in\PP_T(X)}\mu(C).
\end{equation}

\begin{defn}
    A t.d.s. $(X,T)$ is said to have the \textbf{small boundary property} (SBP) if for each $x\in X$ and any neighborhood $U$ of $x$, there exists a neighborhood $V$of $x$ such that $V\subset U$ and $\ocap(\partial V)=0$.
\end{defn} 

\begin{thm}[\cite{kerr2020almost}, Theorem 5.5 (i) (ii) \& \cite{Gutman2017Corrigendum}]\label{thm:kerr}
    A t.d.s. $(X,T)$ has SBP if and only if it admits a zero-dimensional strongly isomorphic extension. 
\end{thm}

The following theorem is deduced from \cite[Theorem 3.3]{L95}.

\begin{thm}\label{thm:finite dienmsion and periodic point=>SBP}
    Let $(X,T)$ be a finite dimensional t.d.s. If $\mathrm{dim}\per(X,T)=0$, then $(X,T)$ has SBP.
\end{thm}
Conversely, one has the following result.
\begin{lem}\label{lem:SBP-implies-periodic-zero-dim}
Let $(X,T)$ be a t.d.s. If 
$(X,T)$ has SBP, then 
$\dim \per(X,T)=0$.
\end{lem}
\begin{proof}
Assume for a contradiction that $(X,T)$ has SBP. We first claim that if $V\subset X$ is open and $\ocap(\partial V)=0$, then
\begin{equation}\label{eq:boundary-avoids-periodic}
    \partial V\cap \per(X,T)=\emptyset.
\end{equation}
Indeed, if $p\in \partial V$ is periodic with period $k$, then 
\[
\mu_p:=\frac1k\sum_{j=0}^{k-1}\delta_{T^j p}\in\mc \PP_T(X)
\]
satisfies $\mu_p(\partial V)\ge 1/k>0$. This contradicts $\ocap(\partial V)=0$.

Now fix $p\in \per(X,T)$ and an open neighborhood $U$ of $p$ in $X$. By SBP, there exists an open set
$V$ such that
\[
p\in V\subset U
\quad\text{and}\quad
\ocap(\partial V)=0.
\]
By \eqref{eq:boundary-avoids-periodic}, we have $\partial V\cap \per(X,T)=\emptyset$. Set $W:=V\cap \per(X,T)$.
Then $W$ is open in  $\per(X,T)$ and $p\in W\subset U\cap \per(X,T)$. Moreover,
the boundary of $W$ in $\per(X,T)$ satisfies
\[
\partial_{\per(X,T)} W \subset \partial V\cap \per(X,T)=\emptyset,
\]
so $W$ is also closed in $\per(X,T)$. Hence $\per(X,T)$ has a clopen basis. Therefore $\per(X,T)$ is zero-dimensional, which is a contradiction.
\end{proof}

\begin{lem}\label{lem:SBPtau-implies-tau-aperiodic}
Let $(X,\R)$ be a fixed-point free topological flow and $\tau\in \R$ so that $\dim \per(X,\tau)=0$, then $(X,\tau)$ is aperiodic.
\end{lem}
\begin{proof}
 As $\dim \per(X,\tau)=0$, $\tau\neq 0$.  Assume for a contradiction $x\in \per(X,\tau)\neq \emptyset$. Let $n\in \N$ so that $(n\tau)x=x$. Denote by $\stab_\R(x)=\{r\in \R|\,rx=x\}$ the  \textit{stabilizer} of $x$ which clearly is a closed subgroup. As  $(X,\R)$ is fixed-point free, $\stab_\R(x)\neq \R$. Thus  $n\tau\in \stab_\R(x)=a\Z$ for some $a\neq 0$. We conclude that the subflow $([0,a]x,\R)$ is isomorphic as a topological flow to the circle of circumference $a$ with the standard $\R$-action by translation. Thus $n\tau=m a$ for some $m\in \Z$. In particular for every $y\in [0,a]x$ $(n\tau) y= y$ which implies $\dim \per(X,\tau)\geq 1$. Q.E.D. 
\end{proof}

\subsection{The small flow boundary property}

 \begin{defn}[\cite{burguet2019symbolic}, Definition 2.3]
Let $(X,\R)$ be a fixed-point free topological flow. A closed cross-section $S$ with injectivity time $\eta$ has a \textbf{small flow boundary}  if $[-\eta,\eta](\pp S)$ is a null set.
 \end{defn}
 For a cross-section $A\subset X$ (with some positive injectivity time),  define the \textbf{counting orbit capacity} of $A$ by\footnote{Note that unlike $\mathrm{ocap}(\cdot)$, $\mathrm{ocap}_\#(\cdot)$ may exceed $1$.} 
 \begin{equation}\label{eq:ocapf}
     \mathrm{ocap}_\#(A):=\lim_{T\to\infty}\frac{1}{T}\max_{x\in X}|\{0\le t\le T:tx\in A\}|,
 \end{equation}
 where   the limit exists and is finite, as $\max_{x\in X}|\{0\le t\le T:tx\in A\}|$ is subadditive (see \cite[P. 4362]{burguet2019symbolic}).

\begin{lem}[\cite{burguet2019symbolic}, Lemma 2.10]
Suppose that $S$ is a closed cross-section with injectivity time $\eta>0$. Then the following are equivalent.
\begin{enumerate}
\item $S$ has a small flow boundary;
\medskip
\item $\mathrm{ocap}_\#(\pp S)=0$.
\end{enumerate}
\end{lem}

\begin{defn}\label{def:SFBP}
A topological flow $(X,\R)$ is said to have  the \textbf{small flow boundary property} (SFBP) if $(X,\R)$ is fixed-point free and for  each $x\in X$ and closed cross-section $S'$ with $x\in \ip(S')$, there exists a closed subset $S\subset S'$ such that $x\in \ip(S)$ and $S$ has a small flow boundary.
\end{defn}

In \cite{gutman2024strongly} it was shown that every finite-dimensional  fixed-point free topological flow and with only countably many periodic orbits satisfies SFBP.
By \cite[Theorem~5]{bowen1972expansive}, an expansive flow has at most countably many periodic orbits, and by \cite[Lemma~1]{bowen1972expansive} it has only finitely many fixed points, each of which is isolated. Since the set $F$ of fixed points is finite and each point of $F$ is isolated, $F$ is open; hence $X_0:=X\setminus F$ is a closed $\R$-invariant subset and therefore compact. 
In particular, the restricted flow $(X_0,\R)$ is fixed-point free. 
Many classical examples fall into this framework. 
For instance, Anosov flows are expansive \cite{anosov1967geodesic}; in particular, the geodesic flow on a compact smooth manifold of negative sectional curvature is Anosov and hence expansive. 
Also Axiom~A flows are expansive when restricted to their nonwandering sets \cite{Bowen1972}.

The following  lemma plays an important role in our construction of small flow-generated entropy factors.
\begin{lem}\label{lem:key2}
Let $S$ be a global closed cross-section with an injectivity time $\eta>0$.
 Assume that $\xp$ has SFBP.   Let $U,V\subset S$ with $\ip U=U$, $\ip V=V$ and $\overline{U}\subset V$. Then there exists $U'\subset S$ with $\ip U'=U'$ such that $U\subset U'\subset V$ and $\ocap_\#(\pp U')=0.$
\end{lem}
\begin{proof}
   Since $\xp$ has SFBP, it follows that for  all $x\in \pp U$, there exists a subset $V_x\subset V$ with $\ip V_x=V_x$ such that $x\in V_x$ and $\ocap_\#(\pp V_x)=0$. Since $\pp U$ is compact, there exist $x_1,\ldots,x_N$ such that 
    $\pp U\subset \cup_{i=1}^NV_{x_i}$. By Lemma \ref{lem:flow boudary} (1), one has 
    \begin{equation}\label{eq:15.37}
        \pp U\subset \cup_{i=1}^N \pp V_{x_i}.
    \end{equation}
    Set 
    \[U':=\cup_{i=1}^NV_{x_i}\cup U\subset S.\]
    Applying Lemma \ref{lem:flow boudary} (1) again and using \eqref{eq:15.37}, we obtain $\pp U'\subset \cup_{i=1}^N\pp V_{x_i}$, which implies that  \[\ocap_\#(\pp U')\le\sum_{i=1}^N\ocap_\#(\pp V_{x_i})=0.\] Furthermore,  by Lemma \ref{lem:flow boudary} (2), we deduce that
    \[U'=\cup_{i=1}^N\ip V_{x_i}\cup \ip U\subset \ip(\cup_{i=1}^NV_{x_i}\cup U)=\ip U',\]
    and hence 
    \[U'=\ip U'.\qedhere\]
\end{proof}

\begin{prop}
    \label{prop:sfbp_suspension_mdim_zero}
Let $(X,T)$ be a t.d.s.\ and
$f:X\to \R_{>0}$ a continuous function. If the suspension flow $(S_fX,\R)$ has
SFBP, then
$\operatorname{mdim}(S_fX,\R)=0$.
\end{prop}
\begin{proof}
Let $(S_1X,\R)$ denote the unit-roof suspension over $(X,T)$. By
\cite[Lemma~2.15]{burguet2019symbolic},
$(X,T)$ has SBP. Hence by By \cite[Theorem~5.4]{LW},$  \operatorname{mdim}(X,T)=0$.
From \cite[Theorem~1.2]{levin2023finitetooneequivariantmapsmean}, it holds $  \operatorname{mdim}(S_1X,\R)
  =
  \operatorname{mdim}(X,T)
  =
  0$.
 By Lemma  \ref{lem:orbit_eq} 
 $(S_fX,\R)$ and $(S_1X,\R)$ are topologically orbit equivalent. Since
 in addition both suspension flows are fixed-point free we may apply
\cite[Theorem~4.1]{JinQiao2026}, to conclude
$
  \operatorname{mdim}(S_fX,\R)=0 
$.
\end{proof}

\subsection{Variational principle for the counting orbit capacity}

In analogy with \eqref{eq123}, we present a characterization of the counting orbit capacity of a closed subset of a global cross-section as the supremum of the measures of this set, taken over a suitable class of measures (see Proposition \ref{prop:characterize of flow capactiy} below). We start by proving some properties of the counting orbit capacity.
\begin{lem}\label{eq:from X to S} Let $(X,\R)$ be a fixed-point free flow, and let $S$ be a global closed cross-section. For any subset $A$ of $S$, there exists $C:=C(S)>0$ such that for all $T>0$, one has 
     \[\max_{x\in X}|\{0\le t\le T:tx\in  A\}|\le\max_{x\in S}\bigl|\{0\le t\le T:tx\in A\}\bigr|
+C.\]
\end{lem}
\begin{proof}
Denote the injectivity time of $S$ by $\eta$. As $S$ is global, there exists $\xi$ such that 
\[[-\xi,\xi]S=X.\]
Fix $x\in X$. We thus have $\tau\in[-\xi,\xi]$
such that $z:=\tau x\in S$.
Then
\[
\{0\le t\le T:tx\in A\}
=
\{-\tau\le t\le T-\tau:t z \in A\}.
\]
Note that
\[
\bigl|\{-\tau\le t\le T-\tau: t z \in A\}\bigr|
\le
\bigl|\{0\le t\le T:tz\in A\}\bigr|+\frac{|\tau|}{2\eta}+1.
\]
Since $|\tau|\le \xi$, taking maximum over $x\in X$ yields the desired inequality.
\end{proof}

\begin{lem}\label{lem:continuity of capacity}
Let $(X,\R)$ be a fixed-point free flow, and let $S$ be a global closed cross-section. Let $A\subset S$ be closed and set $A_\epsilon:=\{x\in S:d(x,A)<\epsilon\}$.
Then
\[
\lim_{n\to\infty}\ocap_\#(A_{1/n})=\ocap_\#(A).
\]
\end{lem}

\begin{proof}
Since $A_{1/n}\downarrow A$ and $\ocap_\#$ is monotone, the limit
$$K:=\lim_{n\to\infty}\ocap_\#(A_{1/n})$$ exists and satisfies $K\ge \ocap_\#(A)$.
We show $\ocap_\#(A)\ge K$.

Fix $T>0$. For each $n\in\N$ choose $x_n\in X$ such that
\begin{equation}\label{eq:202639}
\begin{split}
        \frac1T\bigl|\{0\le t\le T:tx_n\in A_{1/n}\}\bigr|
=&\frac1T\max_{x\in X}\bigl|\{0\le t\le T:tx\in A_{1/n}\}\bigr|\\
\ge& \ocap_\#(A_{1/n})\ge K.
\end{split}
\end{equation}
By compactness of $X$, passing to a subsequence we may assume   that  $x_n\to y\in X$.
Let $k\in\N$ be arbitrary, and $\rho:=1/(2k)$. By uniform continuity of $(t,z)\mapsto tz$ on $[0,T]\times X$,
for all large enough $n\in\N$ we have
\[
\sup_{0\le t\le T} d(tx_n,ty)<\rho
\quad\text{and}\quad 1/n<\rho.
\]
Hence, for such $n\in\N$,
\[
tx_n\in A_{\frac{1}{n}}\ \Longrightarrow\ ty\in A_{\rho+\frac{1}{n}}\subset A_{2\rho}.
\]
Thus, from \eqref{eq:202639}, for all $k\in\N$,
\begin{equation}\label{eq:2026392146}
    \frac1T\bigl|\big\{0\le t\le T:ty\in A_{1/k}\big\}\bigr|=\frac1T\bigl|\{0\le t\le T:ty\in A_{2\rho}\}\bigr|\ge K.
\end{equation}

Since $A=\bigcap_{m\ge1}A_{1/m}$ is closed, $A_{1/m}\subset S$
and  the set $\{t\in[0,T]:tx\in S\}$ is finite, it follows that
\begin{align*}
  \frac{1}{T} \bigl|\{0\le t\le T:ty\in A\}\bigr|
=\lim_{k\to\infty}\frac{1}{T}\bigl|\{0\le t\le T:ty\in A_{1/k}\}\bigr|\overset{\eqref{eq:2026392146}}{\ge} K.
\end{align*}
By subadditivity, $\ocap_\#(A)=\inf_{T>0}\frac1T\max_{x\in X}|\{0\le t\le T:tx\in A\}|\ge K$, as desired.
\end{proof}
As a direct corollary, one has the following lemma.
\begin{lem}\label{lem:key1}
Let $(X,\R)$ be a fixed-point free flow, and let $S$ be a global closed cross-section with an injectivity time.
Let $E\subset S$ be closed and $\epsilon>0$. Then there exists a relatively open set
$U\subset S$ such that $E\subset U$ and
\[
\ocap_\#(U)\le \ocap_\#(E)+\epsilon.
\]
\end{lem}

We are now able to provide a characterization of $\ocap_\#(\cdot)$. For $ r\ge 0$, denote by $\M_{\le r}(S)$ the space of all measures $\mu$ on $S$ with $\mu(S)\le r$.
It follows from the Banach-Alaoglu theorem that $\M_{\le r}(S)$ is compact under the weak* topology for all $ r\ge 0$.

For the reader's convenience, we recall  part of the Portmanteau theorem (see e.g. \cite[Theorem 2.1]{Billingsleybook1999}):
\begin{lem}\label{lem:portmanteau}
Let $(X,d)$ be a metric space and let $\mu_n,\mu$ be Borel probability measures on $X$.
Assume that $\mu_n$ converges to $\mu$. Then
\begin{enumerate}
\item for every closed set $F\subset X$,
\[
\limsup_{n\to\infty}\mu_n(F)\le \mu(F);
\]
\item for every open set $G\subset X$,
\[
\liminf_{n\to\infty}\mu_n(G)\ge \mu(G).
\]
\end{enumerate}
\end{lem}

\noindent Denote $K:=\ocap_\#(S)$ and define $\M_{\le K}(S):=\{\mu\ge0:\mu(S)\le K\}$ as well as
\[\mathcal{OM}_S:=\{\mu\in \M_{\le K}(S):\mu(A)\le \ocap_\#(A),\text{ for  all closed }A\subset S\}.\]
Notice that the set $\mathcal{OM}_S$ is weak* compact in $\M_{\le K}(S)$.
Indeed, the set $\M_{\le K}(S)$ is weak* compact, and for each closed $F\subset S$
the map $\mu\mapsto \mu(F)$ is upper semicontinuous, from the Portmanteau theorem (Lemma \ref{lem:portmanteau} (1)). Equivalently, for all $a\ge0$, $\{\mu:\mu(F)\le a\}$ is weak* compact, in particular, $\{\mu:\mu(F)\le \ocap_\#(F)\}$ is weak* closed. Therefore $\mathcal{OM}_S$, being an intersection of weak* closed sets,
is weak* compact.
\begin{prop}\label{prop:characterize of flow capactiy}
Let $(X,\R)$ be a fixed-point free flow, and let $S$ be a global closed cross-section.
For every closed set $A\subset S$   we have 
\[
\ocap_\#(A)=\max_{\mu\in\mathcal{OM}_S}\mu(A).
\]
\end{prop}

\begin{proof}
By the definition of $\mathcal{OM}_S$, one has $\mu(A)\le \ocap_\#(A)$ for all $\mu\in\mathcal{OM}_S$,
and hence $\sup_{\mu\in\mathcal{OM}_S}\mu(A)\le \ocap_\#(A)$.
Since $\mathcal{OM}_S$ is weak* compact in $\M_{\le K}(S)$ and $\mu\mapsto \mu(A)$ is upper semicontinuous
for fixed closed set $A$, from the Portmanteau theorem (Lemma \ref{lem:portmanteau} (1)), the supremum is attained, so it suffices to find $\mu\in\mathcal{OM}_S$ with
$\mu(A)\ge \ocap_\#(A)$.

Fix a closed $A\subset S$. For each $n\in\N$, by Lemma~\ref{eq:from X to S}, we may choose $x_n\in S$ such that
\begin{equation}\label{eq:202616}
\frac1n\bigl|\{t\in[0,n]:t x_n\in A\}\bigr|
\ge
\frac1n\max_{x\in X}\bigl|\{t\in[0,n]:t x\in A\}\bigr|
-\frac{C}{n}.
\end{equation}
For $n\in\N$ and $B\subset S$, set
\[
H_n(B):=\{t\in[0,n]:t x_n\in B\}.
\]
Letting $n\to\infty$ in \eqref{eq:202616} yields
\begin{equation}\label{eq:H_n(A)}
\limsup_{n\to\infty}\frac{1}{n}\,|H_n(A)|=\ocap_\#(A).
\end{equation}
Also
\begin{equation}\label{eq:H_n(S)}
    \limsup_{n\to\infty}\frac{1}{n}\,|H_n(S)|\le \ocap_\#(S).
\end{equation}

We now construct a measure $\mu^A\in\mc{OM}_S$ from the sequence $\{x_n\}_{n\ge1}$ and the sets $H_n(S)$, $n\in\N$. To this end,
since $S$ has an injectivity time, the set $H_n(S)$ is finite for every $n\in\N$.
Write
\[
H_n(S)=\{t\in[0,n]:t x_n\in S\}=\{t_1^n<\cdots<t_{m_n}^n\},
\]
 and define
\[
\mu^A_n:=\frac1n\sum_{i=1}^{m_n}\delta_{t_i^n x_n}.
\]
Note that,  by \eqref{eq:H_n(S)}, for $n\in\N$ sufficiently large, one has $\mu^A_n\in \M_{\le K+1}(S)$.

Passing to a subsequence, we may assume    that  $\mu^A_n\to\mu^A$ in the weak* topology,
for some $\mu^A\in \M_{\le K+1}(S)$, as $\M_{\le K+1}(S)$ is weak* compact. Furthermore, by the Portmanteau theorem (Lemma \ref{lem:portmanteau} (2)) applied to $S$ which is open in itself, one has
\[\mu^A(S)=\liminf_{n\to\infty}\mu^A_n(S)\le K,\]
which implies that 
\[\mu^A\in \M_{\le K}(S). \]

We now prove    that  $\mu^A(A)\ge \ocap_\#(A)$.
Since $A$ is closed, the Portmanteau theorem (Lemma \ref{lem:portmanteau} (1)) gives
\begin{align*}
    \mu^A(A)\ge \limsup_{n\to\infty}\mu^A_n(A)
=\limsup_{n\to\infty}\frac1n\bigl|H_n(A)\bigr|\overset{\eqref{eq:H_n(A)}}{=} \ocap_\#(A).
\end{align*}

We complete the proof by proving $\mu^A(F)\le \ocap_\#(F)$ for all closed $F\subset S$.
Fix a closed $F\subset S$. For $k\in\N$ set $F_{1/k}:=\{x\in S:d(x,F)<1/k\}$, which is open in $S$.
By the Portmanteau theorem (Lemma \ref{lem:portmanteau} (2)),    we have
\begin{align*}
    \mu^A(F_{1/k})\le \liminf_{n\to\infty}\mu^A_n(F_{1/k})
=&\liminf_{n\to\infty}\frac1n\bigl|H_n(F_{1/k})\bigr|\\
\le& \liminf_{n\to\infty}\frac1n\max_{x\in X}\bigl|\{0\le t\le n:tx\in F_{1/k}\}\bigr|\\
=&\ocap_\#(F_{1/k}).
\end{align*}
Using outer regularity of $\mu^A$ and Lemma~\ref{lem:continuity of capacity},
letting $k\to\infty$ yields $\mu^A(F)\le \ocap_\#(F)$.
As $F\subset S$ was arbitrary, $\mu^A\in\mathcal{OM}_S$.
\end{proof}

\begin{lem}\label{lem:ocapsharp-subadd}
Let $(X,\R)$ be a fixed-point free flow, and let $S$ be a closed cross-section. Given subsets $A_i\subset S$, $i\in\N$, it holds   that
\[
\ocap_\#\Bigl(\bigcup_{i=1}^{\infty}A_i\Bigr)\le \sum_{i=1}^{\infty}\ocap_\#(A_i).
\]
\end{lem}
\begin{proof}
For every $x\in X$ and $T>0$,
\[
\bigl|\{0\le t\le T:\ t x\in \cup_{i\in\N}A_i\}\bigr|\le \sum_{i=1}^{\infty}\bigl|\{0\le t\le T:\ t x\in A_i\}\bigr|.
\]
Taking the maximum over $x$, one   gets 
\begin{align*}
   \max_{x\in X} \bigl|\{0\le t\le T:\ t x\in \cup_{i\in\N}A_i\}\bigr|\le& \max_{x\in X} \left(\sum_{i=1}^{\infty}\bigl|\{0\le t\le T:\ t x\in A_i\}\bigr|\right)\\
   \le&\sum_{i=1}^{\infty}\max_{x\in X} \bigl|\{0\le t\le T:\ t x\in A_i\}\bigr|.
\end{align*}
Dividing by $T$ and letting $T\to\infty$ gives the desired inequality:
\[
\ocap_\#\Bigl(\bigcup_{i=1}^{\infty}A_i\Bigr)
\le
\sum_{i=1}^{\infty}\ocap_\#(A_i).
\qedhere\]
\end{proof}

\subsection{Proofs of Theorems ~\ref{thm:SFBPvsSBP_tau} and \ref{thm:SFBP_implies_mdim0}}\footnote{Precursors to Theorems \ref{thm:SFBPvsSBP_tau} and \ref{thm:SFBP_implies_mdim0} for aperiodic flows were first obtained  by Ruxi Shi and the second author.}

\begin{proof}[Proof of Theorem~\ref{thm:SFBPvsSBP_tau}]
(1) $\Rightarrow$ (2). Since $(X,\R)$ has SFBP, by \cite[Proposition 2.2]{burguet2019symbolic}, it admits a strongly isomorphic extension $\pi:(Y,\R)\rightarrow (X,\R)$  which is a suspension flow over a zero-dimensional
 space. In particular, $Y$ is finite dimensional (see e.g. \cite{MR54240}).  As  $(X,\tau)$ is aperiodic, so is $(Y, \tau)$. Thus Theorem \ref{thm:finite dienmsion and periodic point=>SBP} implies  $(Y, \tau)$ has  SBP. Therefore, by Theorem \ref{thm:kerr}, $(Y, \tau)$ admits a zero-dimensional strongly isomorphic extension $\rho:(Z,T)\rightarrow (Y, \tau)$. It follows that $\pi\circ \rho: (Z,T) \to (X,\tau)$ is strongly isomorphic by Lemma \ref{lem:strong_composition}. Finally,  Theorem \ref{thm:kerr} implies    that  $(X,\tau)$ has SBP.

(2) $\Rightarrow$ (1).   This follows from Lemma~\ref{lem:SBP-implies-periodic-zero-dim} and Lemma \ref{lem:SBPtau-implies-tau-aperiodic}.
\end{proof}

\begin{proof}[Proof of Theorem~\ref{thm:SFBP_implies_mdim0}]
Assume   that $(X,\tau)$ is aperiodic for some $\tau\in \R\setminus\{0\}$. By Theorem~\ref{thm:SFBPvsSBP_tau}, $(X,\tau)$ has SBP.
By \cite[Theorem~5.4]{LW}, $\mdim(X,\tau)=0$. Finally, by Proposition~\ref{prop:mdim time-1}, 
$\mdim(X,\R)=0$.
\end{proof}

\subsection{Examples} 
In general, a topological flow may have SFBP while its $\tau$-discretization fails to have SBP for all $\tau\in \R$; conversely, a flow may have the property that its $\tau$-discretization has SBP for some $\tau>0$ but the flow itself does not have SFBP. We present the following two examples to illustrate these possibilities.
\begin{exam}\label{ex1}  
Let $C$ be the Cantor set. Choose a continuous surjection $f:C\to[1,2]$ and let
$X:=S_f C$ be the suspension space over $(C,\id)$ under the roof function $f$.
Then \textbf{$(X,\R)$ has SFBP, but for every $\tau\in\R\setminus\{0\}$ the $\tau$-discretization
$(X,\tau)$ does not have SBP.}

Indeed, since $C$ is zero-dimensional, the t.d.s.\ $(C,\id)$ has SBP. Hence the suspension
flow $(S_f C,\R)$ has SFBP by \cite[Lemma~2.15]{burguet2019symbolic}. On the other hand,
for every $\tau\in\R$ we have $\dim(\per(X,\tau))=1$, and therefore $(X,\tau)$
does not have SBP by Theorem \ref{thm:SFBPvsSBP_tau}.
\end{exam}

\begin{exam}

Let $(X,T)$ be one of the aperiodic t.d.s. constructed by Dranishnikov and Levin (\cite{DL25}),
which satisfies $\mdim(X,T)=0$ but does not have the marker property and thus, by 
Proposition \ref{prop:appendix}, does not have  SBP either.  Let $Y=S_1X$ be the suspension flow over $(X,T)$ under the constant roof function $1$. Then $(Y,\R)$ is \textbf{an aperiodic flow $(X,\R)$ without SFBP such that the $\tau$-discretization has SBP iff $\tau\in \R\setminus \Q$}. Indeed $(Y,\R)$ is aperiodic because $(X,T)$ is aperiodic.
By \cite[Lemma~2.15]{burguet2019symbolic}, the suspension flow $(S_1X,\R)$ has SFBP
if and only if $(X,T)$ has SBP. Therefore $(Y,\R)$ does not have SFBP. Fix $\tau\in\R\setminus\Q$.
By Lemma~\ref{lem:important observation}, the $\tau$-discretization $(Y,\tau)$ has the marker property.
Note that by Proposition \ref{prop:mdim time-1},
$\mdim(Y,\tau)=|\tau|\mdim(Y,\R)=0$. 
For a t.d.s.\ with the marker property, mean dimension zero implies SBP. 
Consequently, $(Y,\tau)$ has SBP (\cite[Theorem A.3]{gutman2017embedding}).
For $\tau=0$, there is nothing to prove. We now prove that for every $\tau\in\Q\setminus\{0\}$, the system
$(Y,\tau)$ does not have SBP. Suppose, for a contradiction, that there
exists
$
\tau=p/q
$ with $p\in\Z\setminus\{0\},\ q\in\N,$
such that $(Y,\tau)$ has SBP. Since $(Y,\R)$ is aperiodic,
$(Y,\tau)$ is aperiodic. Hence, by
Proposition~\ref{prop:appendix}, $(Y,\tau)$ has the marker property.
By Proposition~\ref{prop:mult. marker} $(Y,p)=(Y,q\tau)$ has the marker property. Using the same proposition again we conclude
 $(Y,1)$ has the marker property. Since the marker property
passes to invariant subsystems, it follows that $(X,T)$ has the marker property, which is a contradiction.
\end{exam}

\begin{defn}
A fixed-point free topological flow $(X,\R)$ is said to have the \textbf{discrete  small flow boundary property (DSBP)}  if for any $\tau\neq 0$, the $\tau$-discretization  has SBP.     
\end{defn}
\begin{rem}
By Lemma \ref{lem:SBP-implies-periodic-zero-dim} and Lemma \ref{lem:SBPtau-implies-tau-aperiodic}, DSBP implies that all $\tau$-discretizations for $\tau\neq 0$ are aperiodic which implies that $(X,\R)$ is aperiodic.     
\end{rem}

\begin{ques}\label{quest:DSBP}
    Does there exist (an aperiodic) topological flow with DSBP but without SFBP?
\end{ques}

\section{Seamless separation}\label{sec:seamless separation}
In this section, we introduce the notion of a \textit{seamless separation} of two distinct points by a countable family of open sets. We then show how such a family can be used to construct a continuous function measurable with respect to the $\sigma$-algebra it generates. This provides a key ingredient for the subsequent construction of small flow-generated entropy factors that separate pairs of distinct points.

\begin{defn}\label{bl}
Let $X$ be a compact metric space and let $x,y\in X$ with $x\neq y$. We will say that a family $\mathcal U=\{U_n\}_{n\ge 0}$ of open subsets of $X$ \textbf{seamlessly separates $x$ from $y$} if $x\in U_0\setminus\bigcup_{n\ge1}U_n$, $y\notin\bigcup_{n\ge0}U_n$ and, for each $p\in\N \cup\{0\}$, the following conditions are fulfilled: 
\begin{enumerate}
	\item[$(A_p)$] $\partial U_p \subset U_{p+1}$,
	\item[$(B_p)$] for all $m>p+1$, $\partial U_p \cap U_m=\emptyset$.
\end{enumerate} 

\end{defn}

\begin{figure}[ht]
\centering
\begin{tikzpicture}[scale=1.15]

\def\r{2.2}
\def\e{0.35}
\def\d{0.18}
\coordinate (O) at (0,0);

\fill[blue!15] (O) circle (\r);

\fill[orange!30, even odd rule]
(O) circle ({\r+\e})
(O) circle ({\r-\e});

\fill[
pattern={Lines[angle=45,distance=2pt,line width=0.15pt]},
pattern color=black,
even odd rule
]
(O) circle ({\r-\e+\d})
(O) circle ({\r-\e-\d});

\fill[
pattern={Lines[angle=-45,distance=2pt,line width=0.15pt]},
pattern color=black,
even odd rule
]
(O) circle ({\r+\e+\d})
(O) circle ({\r+\e-\d});

\draw[line width=0.6pt] (O) circle (\r);

\node at (120:{\r+1.1}) {$\partial U_0$};
\draw[->] (120:{\r+0.9}) -- (120:\r);

\node at (25:{\r+1.4}) {$\partial U_1$};
\draw[->] (25:{\r+1.30}) -- (24:{\r+\e});
\draw[->] (27:{\r+1.26}) -- (34:{\r-\e});

\coordinate (xpt) at (210:0.45*\r);
\fill (xpt) circle (2pt);
\node[below left=2pt] at (xpt) {$x$};

\coordinate (ypt) at ({\r+\e+\d+1.15},1.10);
\fill (ypt) circle (2pt);
\node[right=3pt] at (ypt) {$y$};

\node at (0,0) {$U_0$};
\node at (160:\r) {$U_1$};
\node at (58:{\r-\e}) {$U_2$};
\node at (66:{\r+\e}) {$U_2$};

\end{tikzpicture}
\caption{A schematic picture of $U_0$, $U_1$, and $U_2$.}
\end{figure}
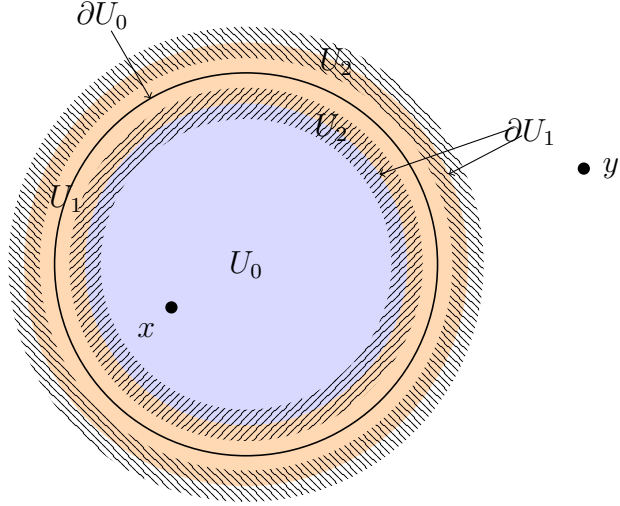

For $n\ge 0$, let $\mathcal P_n$ be the finite partition of $X$ generated by the sets
$U_0,\dots,U_{n}$, i.e.\ the partition into all nonempty sets of the form
\[
\bigcap_{j=0}^{n} A_j,\qquad A_j\in\{U_j,\, U_j^c\}.
\]
Set
\[
\mathcal P_\infty(\U):=\sigma\Bigl(\bigcup_{n\ge0}\mathcal P_n\Bigr),
\]
the $\sigma$--algebra generated by the partitions $\mathcal P_n$.
 Our goal is to separate $x$ from $y$ by a $\P_\infty(\U)$-measurable and continuous function 
$f:X\to[0,1]$, as stated below:

\begin{thm}\label{ltos}
Let $\mathcal U$ be a family of open subsets which seamlessly separate $x$ from $y$. Then there exists a continuous and $\P_{\infty}(\mc U)$-measurable function $f:X\to[0,1]$ such that $f(x)=1,f(y)=0$.
\end{thm}

\begin{rem}\label{rem1}
If $\partial U_n=\emptyset$ for some $n\ge1$, then Theorem~\ref{ltos} is immediate.
Indeed, $\partial U_n=\emptyset$ implies that $U_n$ is clopen.
Since $x\in U_0\setminus\bigcup_{m\ge1}U_m$ and $y\notin\bigcup_{m\ge0}U_m$, the set
$
U_0\setminus \bigcup_{m=1}^n U_m
$
is clopen, belongs to $\mathcal P_\infty(\U)$, contains $x$ and does not contain $y$.
Hence $f:=\mathbf{1}_{\,U_0\setminus \cup_{m=1}^n U_m}$ is continuous, $\mathcal P_\infty(\U)$--measurable, and satisfies $f(x)=1$, $f(y)=0$.
\end{rem}

\begin{rem}\label{rem2}
Note that since, by the condition $(B_{n-1})$, $U_{n+1}$ is disjoint from the boundaries of the atoms $P$ of $\P_{n-1}$, the sets $P\cap U_{n+1}$ are open for all $P\in \P_{n-1}$. 
\end{rem}

\begin{defn}\label{def1}
For each $n\ge 1$, an atom $P$ of $\P_{n-1}$ for which  $P\cap  \partial U_n\neq \emptyset$ is called \textbf{essential}. Otherwise it is called \textbf{inessential}. 
\end{defn}
\begin{rem}\label{rem:equiv. essential}
If $x\in \partial U_n$, then $x\in U_{n+1}$ (since $\partial U_n\subset U_{n+1}$), and hence,
in the partition $\P_{n+1}$, the point $x$ lies in an atom of the form
$
P\cap U_n^{c}\cap U_{n+1}
$
for some $P\in \P_{n-1}$.
\end{rem}
\begin{defn}\label{ess}
We say that the family $\U$ is \textbf{clean} if, for each $n\ge1$, all inessential atoms $P$ of $\P_{n-1}$, satisfy $P\cap U_{n+1}=\emptyset$. 
\end{defn}
\begin{rem}\label{rem:double star}
 By definition, for a clean family $\U$  we have
\[
U_{n+1}
=\bigcup_{\substack{P\in \P_{n-1}\\ P\ \text{essential}}}\bigl(P\cap U_{n+1}\bigr).
\]
\end{rem}
\begin{lem}\label{clean}
If $\U=\{U_n\}_{n=0}^\infty$ is  a family of open subsets which  seamlessly separates $x$ from $y$, then there exists a clean family
$\U'=\{U_n'\}_{n=0}^\infty$ such that for each $n\in\N$,
\[
U_n'\subset U_n,\qquad U_n'\in\P_\infty(\U),\qquad \partial U_n'\subset \partial U_n,
\]
and $\U'$ seamlessly separates $x$ from $y$.
\end{lem}

\begin{proof}
The family $\U'=\{U'_n\}_{n\ge0}$ is defined inductively. We begin by setting $U'_0=U_0$ and $U'_1=U_1$. Suppose that for some $n\ge 1$ we have defined $\P_\infty(\mc U)$-measurable open sets $U'_0,U'_1,\dots,U'_n$ so that the family 
$$
\U^{(n)}=\{U'_0,U'_1,\dots,U'_n, U_{n+1},U_{n+2},U_{n+3},\dots\}
$$ 
satisfies conditions $(A_p)$ and $(B_p)$ for all $p\ge 0$ (this is true for $n=0$ and $n=1$). Let $\P'_{n-1}$ denote the (finite) partition generated by the sets $U'_0,U'_1,\dots,U'_{n-1}$ and their complements. By Remark~\ref{rem2}, for all atoms $P'$ of $\P'_{n-1}$, the set $P'\cap U_{n+1}$ is open.   We let 
\begin{equation}\label{eq:number2}
    U'_{n+1}=\bigcup\bigl\{P'\cap U_{n+1}: P'\in \P'_{n-1}\text{  is essential with respect to }U_n\bigr\}.
\end{equation} 
Clearly, $U'_{n+1}\in \P_{\infty}(\mc U)$. Define 
$$\U^{(n+1)}=\{U'_0,U'_1,\dots,U'_n,U'_{n+1},U_{n+2},U_{n+3},\dots\}.$$
As we only replaced $U_{n+1}$ by $U'_{n+1}$, we only need to verify  the conditions $(A_n)$ and $(A_{n+1})$ for $\U^{(n+1)}$, and  the  condition $ (B_i)$ for $i=0,1, \ldots, n-1$ and $n+1$. Trivially, $U'_{n+1}$ is open and so is $U_{n+1}\setminus U'_{n+1}$  as it equals $\bigcup\{P'\cap U_{n+1}\}$, where $P'$ ranges over all inessential atoms of $\P'_{n-1}$. Thus $U'_{n+1}\subset U_{n+1}$ is relatively clopen in $U_{n+1}$, and hence $\partial U'_{n+1}\subset\partial U_{n+1}$. As $\partial U_{n+1}\subset U_{n+2}$, $\U^{(n+1)}$ satisfies  the  condition $(A_{n+1})$. By construction it also satisfies  the  condition $(A_n)$. As $\partial U'_{n+1}\subset\partial U_{n+1}$, $\U^{(n+1)}$ satisfies   the  condition $(B_{n+1})$. As $U'_{n+1}\subset U_{n+1}$, it satisfies  the condition $(B_i)$ for $i=0,1, \ldots, n-1$.

Define 
$$\U'=\{U_n':n\ge0\}.$$ 
Note that this family  consists of $\P_\infty(\U)$-measurable sets and  is clean by comparing Remark
\ref{rem:equiv. essential} and \eqref{eq:number2}.
The fact  $U'_{n}\subset U_{n}$ for all $n\geq 0$ implies that 
both $x$ and $y$ do not belong to $\bigcup_{n\ge1}U_n'$. Since $U_0'= U_0$, we have
$x\in U_0'\setminus\bigcup_{n\ge1}U_n'$ and $y\notin\bigcup_{n\ge0}U_n'$.
Hence $\U'$ seamlessly separates $x$ from $y$.
\end{proof}
Now we begin to prove Theorem \ref{ltos}.
\begin{proof}[Proof of Theorem \ref{ltos}] 
By Lemma~\ref{clean}, we may assume w.l.o.g. that $\U$ is clean. Moreover, we may further assume that  $\U$ satisfies
\begin{enumerate}
	\item[(C)] for all $n\ge 0$, $ \partial U_n \neq\emptyset$.
\end{enumerate} 
Indeed, if (C) does not hold, then by Remark~\ref{rem1}, the theorem holds trivially.

We will inductively define a \sq\ of functions $(f_n)_{n\ge0}$, which is a Cauchy sequence in the uniform topology. Thus, the function $f$ can be defined as $\lim_{n\to\infty}f_n$.
Let $f_0=1_{U_0}$. Fix some $n\ge 0$. Suppose we have defined $f_n$ which has the following properties:
\begin{enumerate}
	\item[$(1_n)$] $f_n$ is constant on the atoms of $\P_n$ and assumes values of the form $k2^{-n}$ with $k\in\{0,1,\dots,2^n\}$,
	\item[$(2_n)$] the set of discontinuities of $f_n$ equals $\partial U_n$, and 
	\item[$(3_n)$] at every $z\in\partial U_n$ the function has a jump by $2^{-n}$, i.e.,
	$$
	\limsup_{ w\to z}f_n( w)-\liminf_{ w\to z}f_n( w)=2^{-n}.
	$$
	\item[$(4_n)$]$f_n(x)=1, \ \ f_n( y)=0$.
\end{enumerate}
Trivially, these conditions are satisfied for $n=0$.
We will now define $f_{n+1}$. First  define
\begin{equation}\label{eq:def f_n}
    (f_{n+1})_{|U_{n+1}^c}:=(f_n)_{|U_{n+1}^c}.
\end{equation}

This part of the definition alone implies that $f_{n+1}$ fulfills  the condition  $(4_{n+1})$ as $x,y\notin \bigcup_{m\geq 1} U_m$.
It remains\footnote{Note that by {\bf the conditions} $(C)$ and $(A_n)$, $U_{n+1}\neq \emptyset$.} to define $f_{n+1}$ on $U_{n+1}$. By Remark \ref{rem:double star} and Remark \ref{rem2}, $U_{n+1}$ splits into open sets $P\cap U_{n+1}$, where $P$ ranges over the essential atoms of $\P_{n-1}$. Let us focus on one essential atom~$P$ and let $z\in P\cap U_{n+1}$ be a boundary point of $U_n$. This implies that 
\begin{enumerate}
	\item[$(a_n)$] the jump of $f_n$ at $z$, i.e., $\limsup_{w\to z}f_n( w)-\liminf_{w\to z}f_n( w)=2^{-n}$, is achieved for the restriction of $f_n$ to $P\cap U_{n+1}$. 
\end{enumerate}
Observe also that
\begin{enumerate}
	\item[$(b_n)$] $f_n$ assumes constant values on the sets $P\cap U_n\cap U_{n+1}$ and $P\cap U_n^c\cap U_{n+1}$,\footnote{Indeed, the sets $P\cap U_n$ and $P\cap U_n^c$ are atoms of $\P_n$.} 
\end{enumerate}

The conditions $(1_n)$, $(a_n)$ and  $(b_n)$ imply that $f_n$ assumes on $P\cap U_{n+1}$ exactly two values: $k2^{-n}$ and $(k+1)2^{-n}$, for some $k\in\{0,1,\dots,2^n-1\}$. We now define 
\begin{equation}\label{eq:n_average}
\begin{split}
     (f_{n+1})_{|(P\cap U_{n+1})}&:\equiv \frac{1}{2}\big((f_n)_{|P\cap U_n\cap U_{n+1}}+(f_n)_{|P\cap U_n^c\cap U_{n+1}}\big) \\
     &=k2^{-n}+2^{-(n+1)}.
\end{split}
\end{equation}
 Note that this average of two values  has the form $k'2^{-(n+1)}$ where $k'\in\{0,1,\dots,2^{n+1}\}$ is odd. This concludes the definition of $f_{n+1}$.

We need to show that $f_{n+1}$ satisfies   the conditions  $(1_{n+1})$,  $(2_{n+1})$ and  $(3_{n+1})$. On each atom of $\P_{n+1}$ of the form $P'\cap U_{n+1}^c$, where $P'$ is an atom of $\P_{n}$, the function $f_{n+1}$ coincides with $f_n$, hence it is constant (because $f_n$ is constant on $P'$). The other atoms of $\P_{n+1}$ have the form $P\cap U_n\cap U_{n+1}$ or $P\cap U_n^c\cap U_{n+1}$, where $P$ is an essential atom of $\P_{n-1}$. Above, for a fixed essential atom $P\in \P_{n-1}$, the function 
$f_{n+1}$ was defined as the same constant of the form $k2^{-n-1}$ with $k\in\{0,1,\dots,2^{n+1}\}$ on each of these two atoms of $\P_{n+1}$. Therefore $(1_{n+1})$ is fulfilled. 

Next, observe that $f_{n+1}$ has no discontinuities in the interior of $U_{n+1}^c$, because on this set it coincides with $f_n$ whose only discontinuities are at $\partial U_n$, and this boundary is covered by $U_{n+1}$. Likewise, $f_{n+1}$ has no discontinuities within $U_{n+1}$ because this set splits into open sets $P\cap U_{n+1}$ with $P$ being an essential atom of $\P_{n-1}$, and on each of these sets $f_{n+1}$ has been defined as a constant. So, the only discontinuities of $f_{n+1}$ may occur at the boundary points of $U_{n+1}$. 

Let $z\in\partial U_{n+1}$. This point can be approached by points from both $U_{n+1}^c$ and $U_{n+1}$. On $U_{n+1}^c$, $f_{n+1}$ assumes the same values as $f_n$ and these are dyadic rationals of the form $k2^{-n}$. On $U_{n+1}$ the values of $f_{n+1}$ have been defined as dyadic rationals of the form $k'2^{-(n+1)}$ with $k'$ odd. Thus $f_{n+1}$ has a discontinuity at $z$, which completes the proof of   the condition  $(2_{n+1})$. 

In order to prove  the condition  $(3_{n+1})$ we need to look closer at the values of $f_{n+1}$ near $z\in\partial U_{n+1}\subset U_{n+2}$. Let $P$ be an atom of $\P_n$ which contains $z$. By Remark \ref{rem2}  and Remark \ref{rem:double star}, $P\cap U_{n+2}$ is an open neighborhood of $z$, which implies that $z$ is an internal point of $P$. Now, $P$ splits into $P\cap U_{n+1}^c$ and $P\cap U_{n+1}$ which are atoms of $\P_{n+1}$. As $z\in \mathrm{Int}(P)\cap \partial U_{n+1} $, it holds  that 
\[\emptyset\neq \mathrm{Int}(P)\cap U_{n+1} \subset P\cap U_{n+1}\text{ and }\emptyset\neq \mathrm{Int}(P)\cap U_{n+1}^c \subset  P\cap U_{n+1}^c.\]
Since $f_{n+1}$ is constant on these atoms by $(1_{n+1})$, it suffices to show that these constants differ by $2^{-(n+1)}$. Let us thus define
$$
v_c:=(f_{n+1})_{|P\cap U_{n+1}^c}\,\,\,\, v:=(f_{n+1})_{|P\cap U_{n+1}}.
$$
Note that on $P\cap U_{n+1}^c$, $f_{n+1}=f_n$.  As $f_n$ is constant on $P$,
$$
v_c=(f_n)_{P\cap U_{n+1}^c}.    
$$
Write $P=P'\cap V_n$ where  $P'\in\P_{n-1}$ and $V_n=U_n$ or $V_n=U^c_n$. In particular $P\subset P'$ and by \eqref{eq:n_average}, $v=(f_{n+1})_{|(P'\cap U_{n+1})}$ obeys the following formula
\begin{equation*}
\begin{split}
v &= \frac{1}{2} \Big( (f_n)_{|P'\cap U_n\cap U_{n+1}} + (f_n)_{|P'\cap U_n^c\cap U_{n+1}} \Big) \\
\text{(since } \{U_n,U^c_n\} = \{V_n,V^c_n\}) \quad
  &= \frac{1}{2} \Big((f_n)_{|P'\cap V_n\cap U_{n+1}} + (f_n)_{|P'\cap V_n^c\cap U_{n+1}} \Big) \\
\text{(since } P=P'\cap V_n) \quad
  &= \frac{1}{2} \Big((f_n)_{|P\cap U_{n+1}} + (f_n)_{|P'\cap V_n^c\cap U_{n+1}} \Big)\\
\text{(since } f_n \text{ is constant on }P\in \P_n) \quad
  &= \frac{1}{2} \Big(v_c + (f_n)_{|P'\cap V_n^c\cap U_{n+1}} \Big)
\end{split}
\end{equation*}
Moreover by the explanation before \eqref{eq:n_average}, $v_c$ and $(f_n)_{|P'\cap V_n^c\cap U_{n+1}}$ differ by $2^{-n}$. Thus $v$ and $v_c$ differ by $2^{-(n+1)}$ as desired. This ends the proof of   the condition  $(3_{n+1})$.

\begin{figure}[h]
    \centering
    \includegraphics[width=1.0\textwidth]{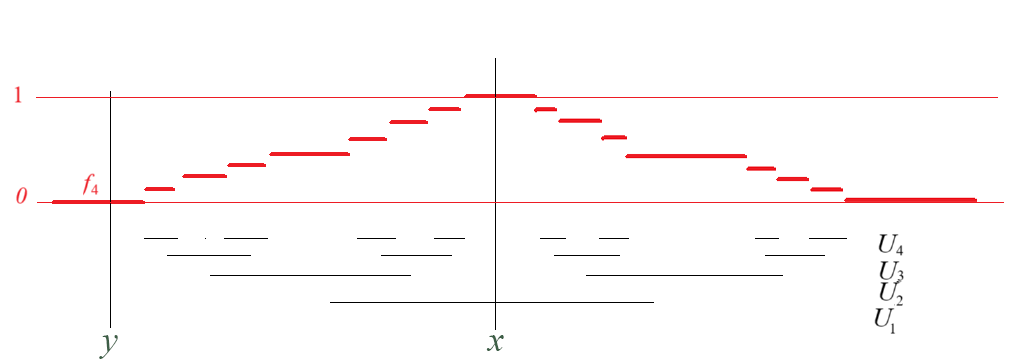} 
    \caption{An illustration of the construction} 
    \label{fig:urysohn}  
\end{figure}

In the last step of the proof we define $f=\lim_{n\to\infty} f_n$.  Observe that this limit exists and is uniform. Indeed, by construction, we have 
\begin{equation}\label{uni}
\|f_n-f_{n+1}\|_{\infty}\le 2^{-(n+1)},\text{ for all }n\in\N,
\end{equation}
where $\|h\|_{\infty}:=\sup_{x\in X}\{|h(x)|\}$. Thus, the \sq\ $(f_n)_{n\ge0}$ is uniformly Cauchy, and hence uniformly convergent to some limit function $f$. Being the pointwise limit of $\P_{\infty}(\mc U)$-measurable functions, $f$ is $\P_{\infty}(\mc U)$-measurable. 
Obviously,   the  condition $(4_n)$ for all $n\in\N$  implies  that $f(x)=1$ and $f(y)=0$. It remains to show that $f$ is continuous. As a consequence of \eqref{uni}, 
we have 
\begin{equation}\label{uni1}
\|f_n-f\|_{\infty}\le 2^{-n},\text{ for all }n\in\N.
\end{equation} 
For brevity, let us use the following notation
$$
\overline f(z) = \limsup_{w\to z}f(w) \text{ \ and \ }\underline f(z) = \liminf_{w\to z}f(w),\text{ for  }z\in X.
$$
Then, for all $z\in X$ and $n\ge 1$, we obtain, using \eqref{uni1},   the condition  $(3_n)$ and~\eqref{uni1}: 
\begin{align*}
0 \le \overline f(z)-\underline f(z)
&= \bigl(\overline f(z)-\overline{f_n}(z)\bigr)
 + \bigl(\overline{f_n}(z)-\underline{f_n}(z)\bigr)
 + \bigl(\underline{f_n}(z)-\underline f(z)\bigr) \\
&\le 3\cdot 2^{-n}.
\end{align*}
Letting $n\to\infty$, it follows that $\overline f(z) - \underline f(z)=0$, and so $f$ is continuous at $z$.
\end{proof}

\begin{defn}\label{def:associated}
    Given a clean family $\U$ which  seamlessly separates two distinct points, we refer to $\{f_n\}_{n=0}^\infty$ and $f$,  constructed  in the proof of Theorem \ref{ltos}, as the \textbf{associated functions} of $\U$.
\end{defn}
\begin{lem}\label{lem:value of fn}
     Given a clean family $\U$ which  seamlessly separates two distinct points, let $\{f_n\}_{n=0}^\infty$ be the associated functions of $\U$. Then for all $n\in\N$, we have
     \begin{equation}\label{eq:2026118}
         f_n(X)\subset\left\{0,\frac{1}{2^n},\ldots,\frac{2^n-1}{2^n},1\right\}.
     \end{equation}
    Furthermore, for every $x\in X$, every \(n\ge0\), every \(j\in\{0,1,\ldots,n\}\), and every odd \(p\in\mathbb N\), the following hold:
    \begin{enumerate}
        \item[$(\mathrm{I}_n)$]  if  $f_n(x)=\frac{p}{2^j}$, then $x\in U_j$;
        \item[$(\mathrm{II}_n)$]   if  $f(x)=\frac{p}{2^j}$, then $x\in U_j\cup U_{j+1}.$
    \end{enumerate}
\end{lem}
\begin{proof}
   The first statement is clear from the construction. We now proceed  to prove the second statement by induction. For $n=0$, this claim is trivial. Assume that the condition $(\mathrm{I}_n)$ holds.
   We now prove the condition $(\mathrm{I}_{n+1})$. 
   
If $j < n + 1$, then by \eqref{eq:n_average},  on $U_{n+1}$, the function $f_{n+1}$ takes values of the form $\frac{k}{2^{n+1}}$ for some odd integer $0 \le k \le 2^{n+1}$. It follows that $x \in U_{n+1}^c$, and hence, by \eqref{eq:def f_n}, we have $f_{n+1}(x) = f_n(x)$. The conclusion then follows from the induction hypothesis.

If $j = n + 1$, then $f_{n+1}(x) \notin  \left\{0, \frac{1}{2^n}, \ldots, \frac{2^n - 1}{2^n}, 1\right\}$. Thus, by \eqref{eq:def f_n} and \eqref{eq:2026118}, we conclude that $x \in U_{n+1}$. Thus,  the condition $(\mathrm{I}_{n+1})$ holds.

Now we establish the condition  $(\mathrm{II}_n)$. Assume that $f(x)=\frac{p}{2^j}$ for some $j\in\{0,1,\ldots,n\}$ and some odd number $p\in\N$. By \eqref{uni1}, 
\[|f_{j+1}(x)-f(x)|\le \|f_{j+1}-f\|_\infty\le \frac{1}{2^{j+1}}.\]
Since by \eqref{eq:2026118},
\[
f_{j+1}(x)\in \left\{0,\frac{1}{2^{j+1}},\ldots,\frac{2^{j+1}-1}{2^{j+1}},1\right\},
\]
it follows that 
$$f_{j+1}(x)\in\{\frac{p}{2^{j}},\frac{2p-1}{2^{j+1}},\frac{2p+1}{2^{j+1}}\}.$$
Applying  the conditions  $(\mathrm{I}_j)$ and $(\mathrm{I}_{j+1})$, we obtain $x\in U_j\cup U_{j+1}$.
\end{proof}

\section{Small flow-generated entropy factors}\label{sec:small entropy factor}
In this section, we prove that if a fixed-point free topological flow $(X,\R)$ has SFBP, then   the  t.d.s. $(X,\tau)$ has the property that any two distinct points in $X$ can be distinguished by a factor with arbitrary small flow-generated entropy. The proof  uses  seamless families and involves the construction of  \textit{dynamical Cantor staircase functions} which are of independent interest (see Definition \ref{defn:dyn_Cantor}).

\subsection{Construction of a clean seamless family}\label{sec:construction of a clean}
 Fix any $w,w'\in X$. Since $(X,\R)$ has no fixed points, by Proposition \ref{prop:existence of complete cross-section}, there is a global closed cross-section
$S$ with injectivity time $\eta>0$. We may assume, w.l.o.g. that $w,w'\in S$. Indeed, if $w'\notin S$, since $S$ is closed the set
$\{t\in\R:\ tw'\in S\}$ is closed and does not contain $0$; hence there exists $\delta>0$
such that $[-\delta,\delta]w'\cap S=\emptyset$.
Replacing $S$ by $S\cup\{w'\}$ and shrinking the injectivity time if necessary (since $S$ is closed and $w'\notin S$ gives a positive time-gap to $S$ along the orbit of $w'$), we may assume that $w'\in S$; we keep the notation $\eta>0$ for the resulting injectivity time.

Our first goal is to construct a clean family $\mathcal U=\{U_n\}_{n\ge0}$ of relatively open subsets of $S$
which seamlessly separates $w$ from $w'$ (in the sense of Definition~\ref{bl}),
and then apply Theorem~\ref{ltos} to obtain a continuous function $f:S\to[0,1]$ with $f(w)=1$ and $f(w')=0$.

Let $a:=2\ocap_\#(\{w\})$.\footnote{If $w$ is aperiodic, then $a=0$. Note that it is possible $a\gg 1$.}
Fix $\epsilon\in(0,\eta)$. If $a>0$, we further assume $\epsilon<a$.
As $\xp$ has SFBP and $\ocap_\#(\{w\})= a/2$, by Lemma \ref{lem:key1} and Lemma \ref{lem:key2}, we may choose a subset $U_0\subset S$ with $U_0= \ip U_0$ such that $w\in U_0$, $w'\notin \overline{U_0}$, 
$
    \ocap_\#(\pp U_0)=0 \text{ and } \ocap_\#(U_0)<a/2+\epsilon/4<a.
$
Recall that $\partial_S$ denotes the boundary in the relative topology of $S$.
By Proposition~\ref{prop:subspace-vs-flow-boundary},
we may  assume that $\overline{U}_0\cap\pp S=\emptyset$, and so $\pp U_0=\partial_S U_0$.

Suppose inductively that we have constructed $U_0,\ldots,U_{n-1}\subset S$ such that 
\begin{enumerate}\label{label:construction}
    \item [($A_{n-1}$)] $\partial_S U_k\subset U_{k+1}$ for $k=0,\ldots,n-2$;
    \item [($B_{n-1}$)] for  all $0\le k\le n-2$, and  $k+2\le m\le n-1$, $\partial_S U_k\cap \overline{U_m}=\emptyset$;
    \item[($C_{n-1}$)]  $\ocap_\#(\partial_S U_{j})=0$, for $j=0,1,\ldots,n-1$;
    \item [($D_{n-1}$)] $\ocap_\#( U_{0})<a$ and $\ocap_\#( U_{j})<\frac{\epsilon}{2^{2j+2}}$, for $j=1,\ldots,n-1$.
\end{enumerate}

From ($C_{n-1}$), $\ocap_\#(\partial_S U_{n-1})=0$, which together with Lemmas \ref{lem:key1} and \ref{lem:key2}, implies that there exists a relatively open subset $U_n\subset S$ with $\overline{U}_n\cap \pp S=\emptyset$ such that $\partial_SU_{n-1}\subset U_n$, $\ocap_\#(\partial_S U_n)=0$, and $\ocap_\#( U_n)<\frac{\epsilon}{2^{2n+2}}$. Thus,  the conditions ($A_{n}$), ($C_{n}$) and ($D_{n}$)  hold. Moreover, by   the condition  ($B_{n-1}$), we can require that for  all $j=0,1,\ldots,n-2$, $\partial_S U_j\cap \overline{U_n}=\emptyset,$ which are exactly the additional conditions which together with   the condition  ($B_{n-1}$) show that  the condition  ($B_{n}$) holds.

Iterating the procedure above, we obtain a family $\U=\{U_n\}_{n=0}^\infty\subset \ip(S)$ which seamlessly separates $w$ from $w'$ (This is witnessed by  the conditions  $(A_n)$ and $(B_n)$ for all $n\in\N$).
By Lemma~\ref{clean}, we may replace $\U$ with a clean family $\U'=\{U_n'\}_{n=0}^\infty$ which seamlessly separates $w$ from $w'$, such that
\[
U_n'\subset U_n
\quad\text{and}\quad
\partial_S U_n'\subset \partial_S U_n
\qquad\text{for all } n\in\N.
\]
In particular, $\U'$ still satisfies   the conditions  $(A_n)$--$(D_n)$ for every $n\in\N$.
For simplicity of notation, we relabel $\U'$ as $\U$ in the remainder of the proof.
 
Let $\{f_n\}_{n\ge0}$ and $f_\epsilon=\lim_{n\to\infty}f_n$ be the associated functions of $\mc U$ (see Definition \ref{def:associated}).
Thus, $f_\epsilon:S\to[0,1]$ is continuous, $\P_\infty(\mathcal U)$-measurable, and satisfies
\[
f_\epsilon(w)=1,\qquad f_\epsilon(w')=0.
\]
Since for all $n\ge0$, $\overline{U_n}\cap \pp S=\emptyset$, we also have $\pp(S)\cap \bigcup_{n\ge0}U_n=\emptyset$. From Lemma \ref{lem:value of fn}, it follows from $f_n(x)>0$  that $x\in \cup_{j=0}^\infty U_j$. Thus, $f_n=0$ on $\pp S$ for all $n$ and therefore
\begin{equation}\label{eq:boundary is 0}
    (f_\epsilon)_{|_{\pp S}}=0.
\end{equation}
\subsection{Dynamical Cantor staircase functions}

Recall the definition of $S$ and $0<\epsilon<\eta$ from the beginning of Subsection \ref{sec:construction of a clean}. Let $L_\eta:=[-\eta,\eta]S$ be the  flow-box associated to $S$ (see Subsection~\ref{sec:flow-box} for the definition).
 Let $D:[0,1]\to[0,1]$ be the Cantor
(devil's) staircase function with $D(0)=0$ and $D(1)=1$ (see Figure \ref{fig:cantor-function}). Denote by $K$ the standard Cantor set.
We now describe its well known-properties.
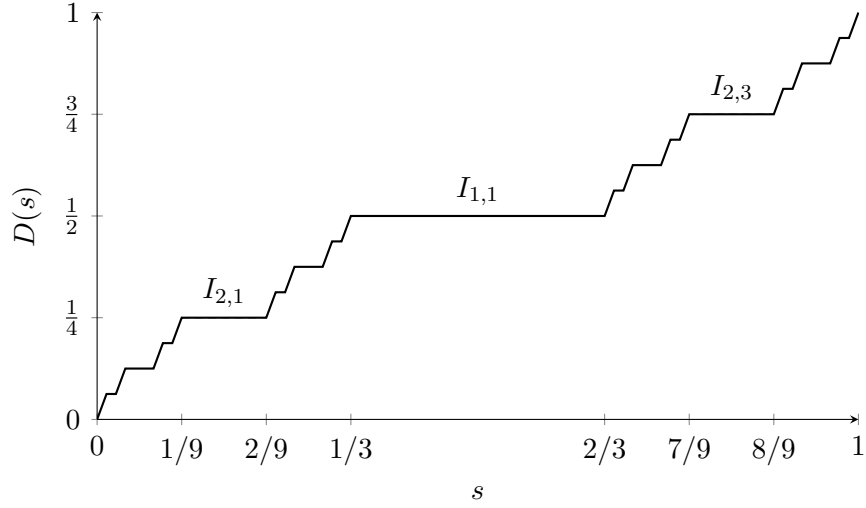
\begin{figure}[t]
\centering
\begin{tikzpicture}
\begin{axis}[
  width=0.92\linewidth,
  height=0.55\linewidth,
  xmin=0, xmax=1, ymin=0, ymax=1,
  axis lines=left,
  xlabel={$s$}, ylabel={$D(s)$},
  xtick={0,1/9,2/9,1/3,2/3,7/9,8/9,1},
xticklabels={$0$,$1/9$,$2/9$,$1/3$,$2/3$,$7/9$,$8/9$,$1$},,
  ytick={0,1/4,1/2,3/4,1},
  yticklabels={$0$,$\frac{1}{4}$,$\frac{1}{2}$,$\frac{3}{4}$,$1$},
  ticklabel style={font=\small},
  label style={font=\small},
  clip=true
]

\addplot[no marks, thick] coordinates {
(0/1,0/1) (1/81,1/16) (2/81,1/16)
(1/27,1/8) (4/81,1/8) (5/81,1/8)
(2/27,1/8) (7/81,3/16) (8/81,3/16)
(1/9,1/4) (10/81,1/4) (11/81,1/4)
(4/27,1/4) (13/81,1/4) (14/81,1/4)
(5/27,1/4) (16/81,1/4) (17/81,1/4)
(2/9,1/4) (19/81,5/16) (20/81,5/16)
(7/27,3/8) (22/81,3/8) (23/81,3/8)
(8/27,3/8) (25/81,7/16) (26/81,7/16)
(1/3,1/2) (28/81,1/2) (29/81,1/2)
(10/27,1/2) (31/81,1/2) (32/81,1/2)
(11/27,1/2) (34/81,1/2) (35/81,1/2)
(4/9,1/2) (37/81,1/2) (38/81,1/2)
(13/27,1/2) (40/81,1/2) (41/81,1/2)
(14/27,1/2) (43/81,1/2) (44/81,1/2)
(5/9,1/2) (46/81,1/2) (47/81,1/2)
(16/27,1/2) (49/81,1/2) (50/81,1/2)
(17/27,1/2) (52/81,1/2) (53/81,1/2)
(2/3,1/2) (55/81,9/16) (56/81,9/16)
(19/27,5/8) (58/81,5/8) (59/81,5/8)
(20/27,5/8) (61/81,11/16) (62/81,11/16)
(7/9,3/4) (64/81,3/4) (65/81,3/4)
(22/27,3/4) (67/81,3/4) (68/81,3/4)
(23/27,3/4) (70/81,3/4) (71/81,3/4)
(8/9,3/4) (73/81,13/16) (74/81,13/16)
(25/27,7/8) (76/81,7/8) (77/81,7/8)
(26/27,7/8) (79/81,15/16) (80/81,15/16)
(1/1,1/1)
};

\addplot[densely dashed] coordinates {(1/3,1/2) (2/3,1/2)};
\node[font=\small, anchor=south] at (axis cs:1/2,1/2) {$I_{1,1}$};

\addplot[densely dashed] coordinates {(1/9,1/4) (2/9,1/4)};
\node[font=\small, anchor=south] at (axis cs:1/6,1/4) {$I_{2,1}$};

\addplot[densely dashed] coordinates {(7/9,3/4) (8/9,3/4)};
\node[font=\small, anchor=south] at (axis cs:5/6,3/4) {$I_{2,3}$};

\end{axis}
\end{tikzpicture}
\caption{The Cantor staircase function $D$.}
\label{fig:cantor-function}
\end{figure}
\begin{lem}\label{lem:cantor-level}
For every $j\ge1$ and every odd integer $p$ with $1\le p\le 2^j-1$,
 there exist disjoint open intervals $I_{j,p}\subset[0,1]\setminus K$ such that
\begin{equation}\label{eq:zero}
    [0,1]\setminus K=\bigcup_{j=1}^\infty \bigcup_{p\text{ odd}}I_{j,p},
\end{equation}
and
\[I_{j,p}=D^{-1}\left(\frac{p}{2^j}\right)\quad\text{ with }\quad\mathrm{Leb}\big(I_{j,p}\big)=3^{-j}.\]
In addition, $D(s)=0$ iff $s=0$, and $D(s)=1$ iff $s=1$.
\end{lem}

Define a continuous function $g_\epsilon:[-\eta,\eta]\rightarrow [0,1]$ by
\[
g_\epsilon(t)=
\begin{cases}
1-D\!\left(\dfrac{|t|}{\epsilon}\right), & t\in[-\epsilon,\epsilon],\\[0.6ex]
0, & t\in[-\eta,\eta]\setminus[-\epsilon,\epsilon].
\end{cases}
\]
Then, by \eqref{eq:zero}, except on the scaled cantor set $K_\epsilon:=\{t\in[-\epsilon,\epsilon]:|t|/\epsilon\in K\}$ of Lebesgue measure zero in $ [-\epsilon, \epsilon] $, the function $ g_\epsilon $ takes values only in the set
\[
\left\{ \frac{j}{2^n} : n \in \mathbb{N},\ j \in \{0,1,\ldots,2^n\} \right\}.
\]

Define $\Theta_\epsilon:X\to[0,1]$ by
\[
\Theta_\epsilon(z)=
\begin{cases}
\min\{f_\epsilon(x_z),\, g_\epsilon(t_z)\}, & z=t_zx_z\in L_\eta,\\[2mm]
0, & z\notin L_\eta,
\end{cases}
\]
where $(t_z,x_z)\in[-\eta,\eta]\times S$ is the unique pair with $z=t_zx_z$. This is well-defined by the injectivity of the flow-box chart.

Moreover, by Remark \ref{rem:equiv def. of flow b}, one has the decomposition
\[
\partial L_\eta = \bigl((-\eta,\eta)\pp S\bigr)\ \sqcup\ (\eta S)\ \sqcup\ (-\eta S).
\]
Since $g_\epsilon$ is continuous on $[-\eta,\eta]$ and vanishes on a neighborhood of $\{\pm\eta\}$, and since $f_\epsilon$ is continuous on $S$ with
$(f_\epsilon)_{|\pp S}\equiv 0$ (see~\eqref{eq:boundary is 0}), it follows that $(\Theta_\epsilon)_{|_{L_\eta}}$ vanishes on a neighborhood of $\partial L_\eta$ in $L_\eta$.
Consequently, $\Theta_\epsilon$ is continuous on $X$, as by definition, it vanishes outside $L_\eta$.

\subsection{An entropy estimate}
Define
\[
F:=\Bigl\{x\in S:\ f_\epsilon(x)\neq \frac{k}{2^n}\ \text{for all }k,n\in\N\cup\{0\}\Bigr\}.
\]
We claim that if $x\in F$, then $x\in U_n$ for infinitely many $n$.
Indeed, assume $x\notin U_n$ for all $n\ge N$.
By \eqref{eq:def f_n}, $f_{n+1}=f_n$ on $U_{n+1}^c$, hence $f_n(x)=f_N(x)$ for all $n\ge N$.
Since $f_n\to f$ uniformly, we have $f_\epsilon(x)=\lim_{n\to\infty}f_n(x)=f_N(x)$.
But by Lemma \ref{lem:value of fn}, $f_N(x)$ is a dyadic rational of the form $k/2^N$, contradicting $x\in F$.
This proves the claim and yields
\begin{equation}\label{eq:relation 1}
    F\subset \bigcap_{N=1}^\infty \bigcup_{n\ge N} U_n.
\end{equation}
By countable subadditivity (Lemma~\ref{lem:ocapsharp-subadd}) together with the monotonicity of
$\ocap_\#$, we obtain 
\begin{equation}\label{eq:oF=0}
    \ocap_\#(F)=0.
\end{equation}

For $j\ge 0$ and odd $p$ with $1\le p\le 2^j-1$, define\footnote{In particular, $A_{0,1}=\{z\in X:\Theta_\epsilon(z)=1\}$.}
\[
A_{j,p}:=\left\{z\in X:\ \Theta_\epsilon(z)=\frac{p}{2^j}\right\},
\qquad
A_0:=\{z\in X:\ \Theta_\epsilon(z)=0\}
\]
and
\[\tilde{A}:=\Bigl\{z\in X:\ \Theta_\epsilon(z)\neq \frac{k}{2^n}\ \text{for all }k,n\in\N\cup\{0\}\Bigr\}.
\]

Note that $z\in \widetilde{A}$ implies that $x_z\in F$ or $t_z\in K_\epsilon$. Therefore,
\[
\widetilde A \subset \bigcup_{t\in[-\epsilon,\epsilon]} tF \ \cup\ \bigcup_{t\in K_\epsilon} tS.
\]
Combining this with~\eqref{eq:relation 1}, we obtain
\begin{equation}\label{eq:relation2}
\widetilde A \subset 
\bigcup_{t\in[-\epsilon,\epsilon]} t\Bigl(\bigcap_{N=1}^\infty \bigcup_{n\ge N} U_n\Bigr)
\ \cup\ 
\bigcup_{t\in K_\epsilon} tS,
\end{equation}
where $\cup_{t\in [-\epsilon,\epsilon]}t(\cup_{N=1}^\infty\cap_{n\ge N}U_n)$ and $\cup_{t\in K_\epsilon}tS$ are analytic (as continuous images of Borel sets in a Polish space), and so   measurable (see e.g. \cite[Theorem 2.9]{G03}).

\begin{thm}\label{thm:dyn_Cantor}
Let $(X,\R)$ be a topological flow with SFBP. The continuous function $\Theta_\epsilon:X\to[0,1]$ constructed above has the following properties:
\begin{enumerate}
    \item 
$\Theta_\epsilon(w)\neq \Theta_\epsilon(w').$
    \item 
    For every flow-invariant probability measure $\mu$, it holds $$\mu(\widetilde A)=0.$$
\end{enumerate}

\end{thm}

\begin{proof}
(1) Since $f_\epsilon(w)=1$, $f_\epsilon(w')=0$ and $w,w'\in S$, it follows that 
   \[\Theta_\epsilon(w)=1\neq 0=\Theta_\epsilon(w').\]
 
(2) By \eqref{eq:relation2} we have
\begin{equation}\label{eq:2026116}
 \mu(\widetilde A)\ \le\
\mu(\cup_{t\in[-\epsilon,\epsilon]} t(\cap_{N=1}^\infty\cup_{n\ge N}U_n))
\ +\
\mu(\cup_{t\in K} tS).   
\end{equation}

We first estimate the first item.
Note that 
\[\mu(\cup_{t\in [-\epsilon,\epsilon]}t(\cap_{N=1}^\infty\cup_{n\ge N}U_n))\le\mu(\cap_{N=1}^\infty\cup_{n\ge N}\cup_{t\in [-\epsilon,\epsilon]}tU_n). \]
By the Borel-Cantelli lemma, to prove that $\mu(\cup_{t\in [-\epsilon,\epsilon]}t(\cap_{N=1}^\infty\cup_{n\ge N}U_n))=0$, it suffices to prove 
\[\sum_{n=1}^\infty \mu(\cup_{t\in [-\epsilon,\epsilon]}tU_n)<\infty.\]
By Lemma~\ref{lem:orbitcap-flowbox} (stated below), with $E=[-\epsilon,\epsilon]$ and $A=U_n\subset S$, we have
\[
\mu(\cup_{t\in[-\epsilon,\epsilon]} tU_n)
\ \le\ \mathrm{Leb}([-\epsilon,\epsilon])\cdot \ocap_\#(U_n)
\ \le\ 2\epsilon\cdot \frac{\epsilon}{2^{2n+2}}.
\]
Hence the above series converges, and thus
\[
\mu(\cup_{t\in[-\epsilon,\epsilon]} t(\cap_{N=1}^\infty\cup_{n\ge N}U_n))=0.
\]

Since $\mathrm{Leb}(K)=0$, Lemma~\ref{lem:orbitcap-flowbox} (stated below), yields 
\[
\mu(\cup_{t\in K} tS)=0.
\]
Then \eqref{eq:2026116} ensures $\mu(\widetilde A)=0$.
\end{proof}

\begin{defn}\label{defn:dyn_Cantor}
Let $(X,\R)$ be a topological flow. A real-valued continuous function on $X$ which has the properties listed in Theorem \ref{thm:dyn_Cantor} is called a \textbf{dynamical  Cantor staircase function}\footnote{A similar definition can be given for $\Z$-t.d.s. One can show that an analogous theorem to Theorem \ref{thm:dyn_Cantor} holds for $\Z$-t.d.s.}.  
\end{defn}

Let
\[
\alpha:=\{A_0\}\cup\{A_{j,p}: j\in\N\cup\{0\},\ p\text{ odd},\ 1\le p\le 2^j-1\}.
\]
Then $\alpha\cup\{\widetilde A\}$ is a measurable partition of $X$.
We now compute the Shannon entropy of the partition $\alpha\cup\{\tilde{A}\}$.
By Theorem~\ref{thm:dyn_Cantor}(2), the set $\widetilde A$ is null.
In particular, for such $\mu$ one has
\[
H_\mu(\alpha)=H_\mu(\alpha\cup\{\widetilde A\}),
\]
so in all entropy computations with the partition $\alpha$ we may ignore $\widetilde A$.

For convenience, we further assume that $\epsilon<4/9$.
\begin{prop}\label{prop:small enough}
Let $\mu$ be a flow-invariant probability measure.
Then it holds that $\mu(A_{0,1})=0$, and for every $j\ge 1$ and every odd $p$ with $1\le p\le 2^{j}-1$, one has
\[
\mu(A_{j,p})\le \frac{(4a+1)\epsilon}{2\cdot3^{j}}.
\]
Moreover,
\[
H_\mu(\alpha)\ \le\ C\epsilon\bigl(1+\log(1/\epsilon)\bigr),
\]
for some constant $C>0$ not depending on $\epsilon$.
\end{prop}

\begin{proof}
For $j\ge0$ and odd $p$ with $1\le p\le 2^{j}-1$, set
\[
E_{j,p}:=\{t\in[-\eta,\eta]: g_\epsilon(t)=p/2^j\},\quad
E_0:=\{t\in[-\eta,\eta]: g_\epsilon(t)=0\},
\]
and
\[
F_{j,p}:=\{x\in S: f_\epsilon(x)=p/2^j\},\quad
F_0:=\{x\in S: f_\epsilon(x)=0\}.
\]

By Lemma~\ref{lem:cantor-level} applied to the  Cantor staircase function $D$ and the definition
$g_\epsilon(t)=1-D(|t|/\epsilon)$ on $[-\epsilon,\epsilon]$ (and $g_\epsilon\equiv 0$ outside), we have
\[
\mathrm{Leb}(E_{0,1})=0,\qquad \mathrm{Leb}(E_0)=2(\eta-\epsilon),\qquad \mathrm{Leb}(E_{j,p})=\frac{2\epsilon}{3^j},\ j\ge1.
\]
By Lemma~\ref{lem:value of fn}, for $j\ge0$ and odd $p$ with $1\le p\le 2^{j}-1$, we have $F_{j,p}\subset U_j\cup U_{j+1}$, and hence by  the condition  $(D_1)$,
\[
\ocap_\#(F_{0,1})\le \ocap_\#(U_0)+\ocap_\#(U_1)\le a+\frac{\epsilon}{16}\]
and for $j\ge 1$, by   the condition  $(D_{j+1})$,
\[\ocap_\#(F_{j,p})\le \ocap_\#(U_j)+\ocap_\#(U_{j+1})\le \frac{\epsilon}{2^{2j+2}}+ \frac{\epsilon}{2^{2(j+1)+2}}<\frac{\epsilon}{2^{2j+1}}.
\]

For the case $j=0$,
since $A_{0,1}\subset \bigcup_{t\in E_{0,1}} tS$ and $\mathrm{Leb}(E_{0,1})=0$,
Lemma~\ref{lem:orbitcap-flowbox} (stated below) yields $\mu(A_{0,1})=0$.

For the case $j\ge1$,
if $\Theta_\epsilon(tx)=p/2^j$, then either $f_\epsilon(x)=p/2^j$ and $g_\epsilon(t)\ge p/2^j$, or $g_\epsilon(t)=p/2^j$ and $f_\epsilon(x)\ge p/2^j$.
Therefore,
\[
A_{j,p}\subset
\Big(\{t\in[-\epsilon,\epsilon]: g_\epsilon(t)\ge p/2^j\}\cdot F_{j,p}\Big)
\cup
\Big(E_{j,p}\cdot \{x\in S: f_\epsilon(x)\ge p/2^j\}\Big).
\]
Applying Lemma~\ref{lem:orbitcap-flowbox} (stated below) twice, we obtain
\begin{align*}
\mu(A_{j,p})
\le\ &
\mathrm{Leb}(\{t\in[-\epsilon,\epsilon]: g_\epsilon(t)\ge p/2^j\})\ \ocap_\#(F_{j,p}) \\
&\quad +\, \mathrm{Leb}(E_{j,p})\ \ocap_\#(\{x\in S: f_\epsilon(x)\ge p/2^j\}).
\end{align*}
Trivially $\mathrm{Leb}(\{t\in[-\epsilon,\epsilon]: g_\epsilon(t)\ge p/2^j\})\le\mathrm{Leb}([-\epsilon,\epsilon])= 2\epsilon$.
Meanwhile, as $\{f_\epsilon\ge p/2^j\}\subset \bigcup_{n=0}^{\infty} U_n\cup F$ (by Lemma \ref{lem:value of fn}), it follows from Lemma \ref{lem:ocapsharp-subadd} and \eqref{eq:oF=0} that
\[
\ocap_\#(\{f_\epsilon\ge p/2^j\})
\le a+\sum_{n=1}^\infty \frac{\epsilon}{2^{2n+2}}+\ocap_\#(F)
= a+\frac{\epsilon}{12}.
\]
Using $\mathrm{Leb}(E_{j,p})=2\epsilon/3^j$ and $\ocap_\#(F_{j,p})\le \epsilon/2^{2j+1}$, we get
\[
\mu(A_{j,p})
\le 2\epsilon\cdot \frac{\epsilon}{2^{2j+1}}
+\frac{2\epsilon}{3^j}\Big(a+\frac{\epsilon}{12}\Big)
=\frac{\epsilon^2}{2^{2j}}+\frac{2a\epsilon}{3^j}+\frac{\epsilon^2}{2\cdot 3^{j+1}}.
\]
Since $\epsilon\le 4/9$, one has $\epsilon^2/2^{2j+1}\le \epsilon/(2\cdot 3^{j+1})$ for all $j\ge1$, hence
\[
\mu(A_{j,p})\le \frac{2a\epsilon}{3^j}+\frac{\epsilon}{ 3^{j+1}}+\frac{\epsilon}{2\cdot 3^{j+1}}
=\frac{(4a+1)\epsilon}{2\cdot3^{j}}.
\]

Finally, we estimate the entropy $H_\mu(\alpha)$.
Let $\varphi(x)=-x\log x$ and for $j\ge1$ set
\[
K:=\frac{4a+1}{2},\quad b_j:=\frac{K\epsilon}{3^{j}},\quad
c_j:=\#\{p:\text{$p$ odd},\,1\le p\le 2^j-1\}=2^{j-1}.
\]
Then $\mu(A_{j,p})\le b_j$ for all such $(j,p)$, and hence
\[
H_\mu(\alpha)=\sum_{A\in\alpha}\varphi(\mu(A))
\le \varphi(\mu(A_0))+\sum_{j\ge1} c_j\,\varphi(b_j).
\]
Moreover,
\[
\sum_{j\ge1}\sum_{p}\mu(A_{j,p})
\le \sum_{j\ge1}c_j b_j
=K\epsilon\sum_{j\ge1}\frac{2^{j-1}}{3^{j}}
=K\epsilon,
\]
so $\mu(A_0)\ge 1-K\epsilon$ and thus, by \cite[(4.1.24), p.~68]{AS64},
\[
\varphi(\mu(A_0))\le K\epsilon+K^2\epsilon^2.
\]
On the other hand,
\begin{align*}
\sum_{j\ge1} c_j\,\varphi(b_j)
&= \sum_{j\ge1}c_j b_j\log\frac1{b_j} \\
&=\log\frac{1}{K\epsilon}\sum_{j\ge1}c_j b_j+\log3\sum_{j\ge1}jc_j b_j\\
&=K\epsilon\log\frac{1}{K\epsilon}+3K\epsilon\log3.
\end{align*}
Putting these estimates together yields
\[
H_\mu(\alpha)\le C\epsilon\bigl(1+\log(1/\epsilon)\bigr),
\]
where $C$ does not depend on $\epsilon$.
\end{proof}

\subsection{Inducing a flow} 
For each $x\in X$, define
\[
(\Theta_\epsilon)_x\in C(\R,[0,1]),\qquad
(\Theta_\epsilon)_x(s):=\Theta_\epsilon(s x),\quad s\in\R.
\]
Equip $C(\R,[0,1])$ with the compact-open topology (the topology induced by uniform convergence on
compact subsets of $\R$), and define the \textbf{hull} of $\Theta_\epsilon$ by 
\[
X_{\Theta_\epsilon}:=\overline{\{(\Theta_\epsilon)_x:x\in X\}}=\{(\Theta_\epsilon)_x:x\in X\}\subset C(\R,[0,1]).
\]

The translation flow on $C(\R,[0,1])$ is given by
\[
(t f)(s):=f(s+t)\qquad \text{ for }f\in C(\R,[0,1]),\ s,t\in\R.
\]
Clearly, $X_{\Theta_\epsilon}$ is $\R$-invariant, and so
 $(X_{\Theta_\epsilon},\R)$  is an $\R$-t.d.s.

Define
\[
\pi_{\Theta_\epsilon}:X\to X_{\Theta_\epsilon},\qquad \pi_{\Theta_\epsilon}(x):=(\Theta_\epsilon)_x.
\]
Then $\pi_{\Theta_\epsilon}$ is continuous and satisfies
\[
\pi_{\Theta_\epsilon}(t x)=t\pi_{\Theta_\epsilon}(x)
\qquad \forall\,x\in X,\ \forall\,t\in\R,
\]
so $\pi_{\Theta_\epsilon}$ is a factor map from $(X,\R)$ onto $(X_{\Theta_\epsilon},\R)$.

Fix $\tau>0$. Define the \textbf{sampling  map}
\[
r_\tau:C(\R,[0,1])\to [0,1]^\Z,\qquad (r_\tau h)(n):=h(n\tau),\quad n\in\Z.
\]
Set
\[
Y_\tau:=Y_\tau(\Theta_\epsilon):=\overline{r_\tau(X_{\Theta_\epsilon})}=r_\tau(X_{\Theta_\epsilon})\subset [0,1]^\Z,
\]
where the closure is taken in the product topology on $[0,1]^\Z$.
Let $T_\tau:Y_\tau\to Y_\tau$ be the shift
\[
(T_\tau y)(n):=y(n+1),\qquad y\in Y_\tau,\ n\in\Z.
\]
Then $(Y_\tau,T_\tau)$ is a t.d.s.
We call $(Y_\tau,T_\tau)$ the \textbf{$\tau$-sampling} of the flow $(X_{\Theta_\epsilon},\R)$. Notice the map $r_\tau:(X_{\Theta_\epsilon},\tau)\to (Y_\tau,T_\tau)$ is a factor map. 

Finally, define
\[
\pi_\tau:X\to Y_\tau,\qquad \pi_\tau:=r_\tau\circ \pi_{\Theta_\epsilon}.
\]
Equivalently, for $x\in X$ and $n\in\Z$,
\[
\pi_\tau(x)(n)=\Theta_\epsilon(n\tau x).
\]
As $r_\tau$ and $\pi_{\Theta_\epsilon}$ are $\tau$-factor maps, it follows that $\pi_\tau$ is a factor map from the $\tau$-discretization $(X,\tau)$ 
onto $(Y_\tau,T_\tau)$.

\subsection{Proof of Theorem \ref{thm:SFBP_small_entropy_factors}}

\begin{proof}
Without loss of generality, we assume that $\tau>0.$
Let $\pi_\tau:(X,\tau)\to (Y_\tau,T_\tau)$ be the factor map defined as above. 
By Theorem~\ref{thm:dyn_Cantor}(1), one has $\Theta_\epsilon(w)\neq \Theta_\epsilon(w').$ Thus,
\begin{equation}\label{eq:sep-F}
\pi_\tau(w)\neq \pi_\tau(w'),
\end{equation}
 because $\pi_\tau(x)(0)=\Theta_\epsilon(x)$ for all $x\in X$.

Let $\mu\in\mc \PP_\R(X)$  and set $\nu:=(\pi_\tau)_*\mu\in\PP_{T_\tau}(Y_\tau)$.
Consider the following countable partition of $Y_\tau$ (with respect to $\nu$):
\[
\beta:=\Big\{\widetilde A_0\Big\}\ \cup\
\Big\{\widetilde A_{j,p}:\ j\ge1,\ p\ \text{odd},\ 1\le p\le 2^j-1\Big\},
\]
where
\[
\widetilde A_0:=\{y\in Y_\tau:\ y(0)=0\},
\]
and for $j\ge1$ and odd $p$,
\[
\widetilde A_{j,p}:=\{y\in Y_\tau:\ y(0)=p/2^j\}.
\]
Note that $\{\widetilde{A}_{j,p}\}$ are pairwise disjoint as 
$\frac{p}{2^j}=\frac{p'}{2^{j'}}$ implies that $j=j'$ and $p=p'$. Moreover,
\[
\nu\Big(\widetilde A_0\cup \bigcup_{j\ge0}\bigcup_{p}\widetilde A_{j,p}\Big)
=\mu\Big(A_0\cup \bigcup_{j\ge0}\bigcup_{p}A_{j,p}\Big)
=1.
\]

\begin{claim}\label{cl:1}
    The partition $\beta$ is a generator of $(Y_\tau,\nu,T_\tau)$.
\end{claim}
\begin{proof}[Proof of Claim \ref{cl:1}]Let
$
A:=\{0\} \cup \{{p}/{2^j}: j\ge 1,\ p\ \text{odd},\ 1\le p\le 2^j-1\}.
$
By the discussion above,
$
\nu\bigl(Y_\tau\cap A^{\Z}\bigr)=1.
$
Now the claim follows trivially.
\end{proof}

Since $\beta$ is a generator, Kolmogorov--Sinai's theorem \cite{sinai2007metric} (see also \cite[Theorem 4.2.2]{D11}) yields
\[
\mathrm h_\nu(T_\tau)= \mathrm h_\nu(T_\tau,\beta).
\]
Moreover, $\pi_\tau^{-1}\beta=\alpha\pmod\mu$, and 
hence $\mathrm h_\nu(T_\tau,\beta)=\mathrm{h_{\mu}}(\tau,\alpha)$.
Therefore, by Proposition~\ref{prop:small enough}, 
\[
\mathrm h_\nu(T_\tau)=\mathrm{h_{\mu}}(\tau,\alpha)\le H_\mu(\alpha)
\le  C\epsilon\bigl(1+\log(1/\epsilon)\bigr).
\]
Since $\mu\in\mc \PP_\R(X)$ is arbitrary, it follows that 
\[h_{\R}(\pi_\tau)\le C\epsilon\bigl(1+\log(1/\epsilon)\bigr). \]
Now choose $\epsilon>0$ so small that
$C\epsilon\bigl(1+\log(1/\epsilon)\bigr)<\delta$, which
together with \eqref{eq:sep-F} completes the proof of Part (1) of the theorem. To prove Part (2), fix a compatible metric $d$ on $X$.
For each $n\in\N$, let
\[
K_n:=\{(x,x')\in X\times X:\ d(x,x')\ge 2^{-n}\}.
\]
Then $K_n$ is a compact subset of $X\times X\setminus \Delta_X$.
We first construct, for each $n\in\N$, a factor map
\[
\Pi_n:(X,\tau)\to (Z_n,S_n)
\]
such that
$h_{\R}(\Pi_n)<\infty$
and every fibre of $\Pi_n$ has diameter $<2^{-n}$.
To this end, fix $n\in\N$. For each $(x,x')\in K_n$, by Theorem~\ref{thm:SFBP_small_entropy_factors} (1),
there exist a $\Z$-t.d.s.\ $(Y_{x,x'},S_{x,x'})$ together with a factor map
\[
\pi_{x,x'}:(X,\tau)\to (Y_{x,x'},S_{x,x'})
\]
such that
$h_{\R}(\pi_{x,x'})<1 \text{ and }
\pi_{x,x'}(x)\neq \pi_{x,x'}(x')$.
Since $\pi_{x,x'}(x)\neq \pi_{x,x'}(x')$, there exist disjoint open sets
$U_{x,x'},V_{x,x'}\subset Y_{x,x'}$ such that
\[
\pi_{x,x'}^{-1}(U_{x,x'})\times \pi_{x,x'}^{-1}(V_{x,x'})
\]
is an open neighborhood of $(x,x')$ in $X\times X$.
These neighborhoods cover $K_n$, and $K_n$ is compact. Thus there exist finitely many pairs
$
(x_{n,1},x'_{n,1}),\dots,(x_{n,m_n},x'_{n,m_n})$,
such that
\[
K_n\subset \bigcup_{i=1}^{m_n}
\Bigl(\pi_{n,i}^{-1}(U_{n,i})\times \pi_{n,i}^{-1}(V_{n,i})\Bigr),
\]
where
\[
\pi_{n,i}:=\pi_{x_{n,i},x'_{n,i}}:(X,\tau)\to (Y_{n,i},S_{n,i}).
\]

Now define the product map
\[
\widetilde\Pi_n:=(\pi_{n,1},\dots,\pi_{n,m_n}):X\to Y_{n,1}\times\cdots\times Y_{n,m_n}.
\]
Let
$Z_n:=\widetilde\Pi_n(X)$,
and let $S_n$ be the restriction of the product map to $Z_n$.
Then we have a (dynamical) factor map 
\[
\Pi_n:(X,\tau)\to (Z_n,S_n).
\]
 By Lemma~\ref{lem:fg-subadditive} proven below,
\[
h_{\R}(\Pi_n)\le \sum_{i=1}^{m_n} h_{\R}(\pi_{n,i})<m_n<\infty.
\]

We claim that every fibre of $\Pi_n$ has diameter $<2^{-n}$.
Indeed, suppose $\Pi_n(x)=\Pi_n(x')$. Then $\pi_{n,i}(x)=\pi_{n,i}(x')$ for all $i$.
If $d(x,x')\ge 2^{-n}$, then $(x,x')\in K_n$, so for some $i$,
\[
(x,x')\in \pi_{n,i}^{-1}(U_{n,i})\times \pi_{n,i}^{-1}(V_{n,i}),
\]
hence
\[
\pi_{n,i}(x)\in U_{n,i},
\qquad
\pi_{n,i}(x')\in V_{n,i},
\]
contradicting $\pi_{n,i}(x)=\pi_{n,i}(x')$.
Thus, each fibre of $\Pi_n$ has diameter $<2^{-n}$.
Next, define
\[
\rho_n:=(\Pi_1,\dots,\Pi_n):X\to W_n:=\rho_n(X)\subset Z_1\times\cdots\times Z_n,
\]
and let $T_n$ be the restriction of $S_1\times\cdots\times S_n$ to $W_n$.
Then $\rho_n:(X,\tau)\to (W_n,T_n)$ is a factor map. Again by Lemma~\ref{lem:fg-subadditive},  we have
\[
h_{\R}(\rho_n)\le \sum_{j=1}^n h_{\R}(\Pi_j)
<
\sum_{j=1}^n m_j<\infty.
\]
Since $\rho_n$ refines $\Pi_n$, every fibre of $\rho_n$ has diameter $<2^{-n}$.
For each $n$, let
$q_n:W_{n+1}\to W_n$ be the natural factor map obtained by forgetting the last coordinate.
Then
$\rho_n=q_n\circ \rho_{n+1}$,
so we obtain an inverse system
\[
(W_1,T_1)\xleftarrow{q_1}(W_2,T_2)\xleftarrow{q_2}(W_3,T_3)\xleftarrow{}\cdots
\]

Define
\[
\rho:X\to \varprojlim (W_n,T_n),
\qquad
\rho(x):=(\rho_n(x))_{n\in\N}.
\]
Because the fibres of $\rho_n$ have diameters tending to $0$, the map $\rho$ is injective.
Since $X$ is compact and the inverse limit is Hausdorff, $\rho$ is a homeomorphism onto its image.
It remains to prove surjectivity. Let $(w_n)_{n\in\N}\in \varprojlim (W_n,T_n)$.
For each $n$, set
$F_n:=\rho_n^{-1}(w_n)$.
Then each $F_n$ is a nonempty compact subset of $X$ with diameter less than $2^{-n}$, and
$F_{n+1}\subseteq F_n$ for all $n\in\N$, because $\rho_n=q_n\circ \rho_{n+1}$.
Thus,
$\bigcap_{n=1}^\infty F_n$
consists of exactly one point, say $x\in X$.
Then $\rho_n(x)=w_n$ for all $n$, so $\rho(x)=(w_n)_{n\in\N}$.
Therefore $\rho$ is surjective.

\end{proof}

\begin{lem}[product factor map inequality]\label{lem:fg-subadditive}
Let $(X,\R)$ be a topological flow, fix $\tau\in\R$, and let
\[
\pi_i:(X,\tau)\to (Y_i,S_i),\qquad i=1,\dots,m,
\]
be factor maps.
Define the product  map
\[
\pi:=(\pi_1,\dots,\pi_m):X\to Y_1\times\cdots\times Y_m,
\]
and let
\[
W:=\pi(X)\subset Y_1\times\cdots\times Y_m.
\]
Let $T$ be the restriction of $S_1\times\cdots\times S_m$ to $W$.
Then
\[
h_{\R}(\pi)\le \sum_{i=1}^m h_{\R}(\pi_i).
\]
\end{lem}

\begin{proof}
Fix $\mu\in \PP_\R(X)$ and set
\[
\nu:=\pi_*\mu\in \PP_T(W),
\text{ and }
\nu_i:=(\pi_i)_*\mu\in \PP_{S_i}(Y_i),~i=1,2,\ldots,m.
\]
For each $i\in\{1,2,\ldots,m\}$, let $\rho_i:W\to Y_i$ be the coordinate projection. Then
\[
\rho_i\circ \pi=\pi_i
\quad\text{and}\quad
(\rho_i)_*\nu=\nu_i,\qquad i=1,2,\ldots,m.
\]

Since each $\rho_i:(W,T)\to (Y_i,S_i)$ is a factor map, it follows from \cite[Fact 4.4.3]{D11} that
\[
h_\nu(T)\le \sum_{i=1}^m h_{(\rho_i)_*\nu}(S_i)
=\sum_{i=1}^m h_{\nu_i}(S_i)
=\sum_{i=1}^m h_{(\pi_i)_*\mu}(S_i).
\]
Taking the supremum over all $\mu\in \PP_\R(X)$ gives
\[
h_{\R}(\pi)\le \sum_{i=1}^m h_{\R}(\pi_i).\qedhere
\]
\end{proof}

Theorem~\ref{thm:SFBP_small_entropy_factors} (1) provides, for each fixed $\tau>0$, a factor
$\pi_\tau:(X,\tau)\to(Y_\tau,T_\tau)$ that separates any prescribed pair of points and has arbitrarily small flow-generated entropy.
It is natural to ask whether the following strengthening holds:
\begin{ques}
Given $\delta>0$ and distinct points $w\neq w'$ in a fixed-point free flow $(X,\R)$,
does there exist a flow factor map $\pi:(X,\R)\to (Y,\R)$ such that $\pi(w)\neq \pi(w')$ and
\[
\mathrm{h_{{top}}}(Y,\R)<\delta?
\]
\end{ques}

We remark that we are able to construct an example which shows that our method does not prove the existence of such flow factors.

\subsection{Estimating measures through $\ocap_\#$}
\begin{lem}\label{lem:orbitcap-flowbox}
Let $(X,\R)$ be a topological flow, let $S$ be a closed cross-section with injectivity time $\eta>0$, and let $\mu$ be a flow-invariant probability measure.
For any Borel set $A\subset S$ and any Borel set $E\subset[-\eta,\eta]$, one has
\[
\mu\left(\cup_{t\in E} tA\right)\ \le\ \mathrm{Leb}(E)\cdot \ocap_\#(A).
\]
\end{lem}
\begin{proof}
By  a standard argument involving the ergodic decomposition, we may assume, w.l.o.g., that $\mu$ is ergodic.

Fix $A\subset S$ and $E\subset[-\eta,\eta]$.
For $x\in X$ set $$R_A(x):=\{s\in\R: sx\in A\}$$ and for $L>0$ set
\[
N_A(x;L):=\left|R_A(x)\cap[-\eta,L+\eta]\right|.
\]
Note that
\[
sx\in \bigcup_{t\in E} tA
~\Longleftrightarrow~
\exists\,t\in E\ \text{ with }\ (s-t)x\in A
~\Longleftrightarrow~
s\in R_A(x)+E.
\]
Thus,
\begin{align*}
    \int_0^L \mathbf 1_{\cup_{t\in E}tA}(sx)ds&\le \mathrm{Leb}\left((R_A(x)+E)\cap[0,L]\right)\\
    &\le\sum_{r\in R_A(x)\cap[-\eta,L+\eta]}\mathrm{Leb}(r+E)\\
&=
\mathrm{Leb}(E)\cdot N_A(x;L).
\end{align*}

Dividing by $L$ and taking $\limsup_{L\to\infty}$ gives,  for every $x\in X$,
\[
\limsup_{L\to\infty}\frac1L\int_0^L \mathbf 1_{\bigcup_{t\in E}tA}(sx)\,ds
\ \le\ \mathrm{Leb}(E)\cdot \ocap_\#(A),
\]
where $\ocap_\#(A)$ is defined in~\eqref{eq:ocapf}.
 
By the pointwise ergodic theorem for flows, the limit
$$\lim_{L\to\infty}\frac1L\int_0^L \mathbf 1_{\cup_{t\in E}tA}(sx)\,ds$$
is equal to $\mu(\cup_{t\in E}tA)$ for $\mu$-a.e. $x\in X$.
Therefore \begin{equation*}
    \mu(\cup_{t\in E}tA)\le \mathrm{Leb}(E)\cdot\ocap_\#(A).\qedhere
\end{equation*}
\end{proof}
As a corollary, we have the following:
\begin{cor}\label{cor:ocapsharp-null}
Let $(X,\R)$ be a topological flow, let $S$ be a closed cross-section with injectivity
time $\eta>0$, and let $A\subset S$ be closed. Then $\ocap_\#(A)=0$ if and only if  $[-\eta,\eta]A$  is a null set. 
\end{cor}

\begin{proof}
One direction follows directly from Lemma \ref{lem:orbitcap-flowbox}. For the reverse implication, 
suppose for contradiction that
$[-\eta,\eta]A$  is null but $\ocap_\#(A)>0$.
Set
\[
c:=\ocap_\#(A)>0,
\]
and fix any $\rho\in(0,\eta)$.

By the definition of $\ocap_\#(A)$, there exist $T_n\to\infty$, as $n\to\infty$, and $\{x_n\}_{n\in\N}\subset X$
such that
\begin{equation}\label{eq:2026319}
\lim_{n\to\infty}\frac{N_n}{T_n}= c, \text{ where }N_n:=\bigl|\{0\le t\le T_n: tx_n\in A\}\bigr|.
\end{equation}
Fix $n\in\N$ and write these hitting times (i.e., the elements of $N_n$) in increasing order as
\[
0\le t^{(n)}_1<\cdots<t^{(n)}_{N_n}\le T_n .
\]

Since $A\subset S$ and $S$ has injectivity time $\eta$, any two distinct hitting
times are separated by more than $2\eta$.
Therefore the intervals
\[
[t^{(n)}_i-\rho,\ t^{(n)}_i+\rho],\qquad i=1,\dots,N_n,
\]
are pairwise disjoint.

Notice that, whenever $\rho\le t^{(n)}_i\le T_n-\rho$, for every
$s\in[-\rho,\rho]$, we have
\[
(t^{(n)}_i+s)x_n\in[-\rho,\rho]A.
\]
Thus,
\[
\mathrm{Leb}\bigl(\{0\le t\le T_n: tx_n\in[-\rho,\rho]A\}\bigr)
\ge 2\rho\,(N_n-2).
\]

Define the following probability measures
\[
\mu_n:=\frac1{T_n}\int_0^{T_n}\delta_{tx_n}\,dt.
\]
Then,
\[
\mu_n([-\rho,\rho]A)
=
\frac1{T_n}\mathrm{Leb}\bigl(\{0\le t\le T_n: tx_n\in[-\rho,\rho]A\}\bigr)
\ge
\frac{2\rho\,(N_n-2)}{T_n}.
\]

Passing to a subsequence if necessary, we may assume that $\mu_n\to\mu$ in the
weak$^\ast$ topology. Clearly, $\mu$ is an $\R$-invariant probability
measure.  Since $A$ is closed and $[-\rho,\rho]$ is compact, the set $[-\rho,\rho]A$ is
compact, hence closed. Therefore, by Portmanteau theorem (Lemma \ref{lem:portmanteau} (1)), the following holds:
\begin{align*}
    \mu([-\rho,\rho]A)\ge \limsup_{n\to\infty}\mu_n([-\rho,\rho]A)
\ge
2\rho \limsup_{n\to\infty}\frac{N_n-2}{T_n}
=
2\rho c
>0.
\end{align*}
This contradicts the assumption that $[-\eta,\eta]A$ is null, as $[-\rho,\rho]A\subset[-\eta,\eta]A$.
\end{proof}
\appendix
\section{SBP implies  the marker property}
\begin{prop}\label{prop:appendix}
    An aperiodic system $(X,T)$ with SBP has the marker property.
\end{prop}
\begin{proof}
    Let $N$ be a natural number. Since $(X, T)$ is aperiodic and has SBP, there exists an open cover $(U_j)_{j=1}^K$ such that 
    \begin{itemize}
        \item $U_k\cap T^{-n}U_k=\emptyset$ for $1\le n\le N$.
        \item $\ocap(\partial U_k)=0$.
    \end{itemize}
    We define open sets $V_1\subset V_2 \subset \cdots \subset V_K$ as follows:
    \begin{itemize}
        \item $V_1=U_1$.
        \item $V_{k+1}=V_k\cup \left( U_{k+1}\setminus \cup_{i=-N}^N T^i\overline{V_k} \right)$.
        \item $V=V_K$.
    \end{itemize}
    It is easy to check that $\ocap(\partial V_k)=0$ for $1\le k\le K$ and 
    $$
    V\cap T^{-n}V=\emptyset, \text{ for any } 1\le n\le N. 
    $$
    We will show $X=\cup_{n\in \Z} T^{n} V$. This  will imply that $(X,T)$ has the marker property.
    
    Suppose $x\in X\setminus \cup_{n\in \Z} T^{n} V$. Then $x\in U_{k_0}$ for some $2\le k_0\le K$. On the other hand, $x\notin V_{k_0}=V_{k_0-1}\cup \left( U_{k_0}\setminus \cup_{i=-N}^N T^{i}\overline{V_{k_0-1}} \right)$. So 
    $$
    x\notin U_{k_0} \setminus \cup_{i=-N}^N T^{ i}\overline{V_{k_0-1}}.
    $$
    This implies that 
    $$
    x\in \cup_{i=-N}^N T^{ i}\overline{V_{k_0-1}}.
    $$
    But $x\notin \cup_{i=-N}^N T^{i}{V_{k_0-1}}$. Thus $x\in \cup_{i=-N}^N T^{ i}{\partial V_{k_0-1}}$. In particular,
    $$
    x\in  \cup_{k=1}^K\cup_{i=-N}^N T^{i}{\partial V_{k}}.
    $$
    Since $X\setminus \cup_{n\in \Z} T^{n} V$ is $T$-invariant, we obtain that $T^mx\in X\setminus \cup_{n\in \Z} T^{n} V$ for each $m\in \Z$ and consequently
    $$
    \ocap \left( \cup_{k=1}^K\cup_{i=-N}^N T^{i}{\partial V_{k}} \right)=1.
    $$
    However since $\ocap(\partial V_k)=0$ for $1\le k\le K$, we have 
    $$
    \ocap \left( \cup_{k=1}^K\cup_{i=-N}^N T^{i}{\partial V_{k}} \right)=0.
    $$
    This is a contradiction.
\end{proof}

\section{Marker property and passage between iterates}
\begin{prop}\label{prop:mult. marker}
Let $(X,T)$ be a t.d.s. For $n\ge 1$, $(X,T)$ has the marker property if and only if $(X,T^n)$ has the marker property.

\end{prop}

\begin{proof}
The implication $(X,T)\Rightarrow (X,T^n)$ is immediate. We prove the converse. Set $S:=T^n$ and fix $N\ge 1$. Let $L>N$. Since $(X,S)$ has the marker property, there exists an open set $V\subset X$ such that
\[
V\cap S^{-q}V=\emptyset \quad (0<|q|<L), 
\qquad
X=\bigcup_{m\in\mathbb Z} S^mV.
\]
By compactness, choose a finite $F\subset\mathbb Z$ with $X=\bigcup_{m\in F}S^mV$. Let $a:X\to[0,1]$ be continuous with $V=\{a>0\}$ and define
\[
b(x):=\max_{m\in F} a(S^{-m}x).
\]
Then $b(x)>0$ for all $x$, hence $\beta:=\min b>0$. Fix $0<\eta<\beta$ and define
\[
U:=\{x:\ a(x)>\eta,\ a(T^i x)<\eta\ \text{for }1\le i<N\}.
\]

If $x\in U\cap T^{-i}U$ for some $1\le i<N$, then $a(T^i x)$ is both $>\eta$ and $<\eta$, a contradiction. Hence $U\cap T^{-i}U=\emptyset$ for $1\le i<N$.

For $x\in X$, set $A_x:=\{k\in\mathbb Z:\ a(T^k x)>\eta\}$. Since $b(T^r x)\ge\beta$, there exist $m\in F$ such that
\[
a(S^{-m}T^r x)=a(T^{r-nm}x)>\eta,
\]
so $r-nm\in A_x$. Thus every translate $r-nF$ meets $A_x$, hence $A_x$ is syndetic and in particular nonempty.

If $k,\ell\in A_x$ and $k\equiv \ell\ (\mathrm{mod}\ n)$, write $\ell-k=nq$. If $0<|q|<L$, then
\[
T^k x,\ T^\ell x=S^q(T^k x)\in V,
\]
contradicting $V\cap S^{-q}V=\emptyset$. Hence $|\ell-k|\ge nL$.

Call $k_0<\cdots<k_m$ an $N$-string in $A_x$ if $k_{j+1}-k_j<N$. Then every $N$-string has at most $n$ elements, since otherwise two indices are congruent mod $n$ and violate the previous bound.

Let $C\subset A_x$ be a maximal $N$-string and set $j=\max C$. By maximality, $a(T^{j+i}x)<\eta$ for $1\le i<N$, while $a(T^j x)>\eta$, so $T^j x\in U$. Hence $x\in \bigcup_{j\in\mathbb Z}T^jU$. Since $x$ was arbitrary, $U$ is an open $N$-marker for $T$.
\end{proof}

\bibliographystyle{alpha}

\bibliography{universal_bib}

\end{document}